%% file: comprelie_libres_vf.tex
\documentclass[11pt,twoside,a4paper]{article}
\usepackage{color,amscd,amsmath,amssymb,amsthm,latexsym,stmaryrd,hyperref}
\usepackage{shuffle}
\usepackage[latin1]{inputenc}
\usepackage[T1]{fontenc}   
\usepackage[english]{babel}
\input{arbres.tex}

\input{hyperarbres}

\input{xy}
\xyoption{all}

\title{Cofree Com-PreLie algebras}
\date{}
\author{Lo\"\i c Foissy\\ \\
{\small \it Fédération de Recherche Mathématique du Nord Pas de Calais FR 2956}\\
{\small \it Laboratoire de Mathématiques Pures et Appliquées Joseph Liouville}\\
{\small \it Université du Littoral Côte dOpale-Centre Universitaire de la Mi-Voix}\\ 
{\small \it 50, rue Ferdinand Buisson, CS 80699,  62228 Calais Cedex, France}\\ \\
{\small \it Email: foissy@univ-littoral.fr}}

\theoremstyle{plain}
\newtheorem{theo}{Theorem}[section]
\newtheorem{lemma}[theo]{Lemma}
\newtheorem{cor}[theo]{Corollary}
\newtheorem{prop}[theo]{Proposition}
\newtheorem{defi}[theo]{Definition}

\theoremstyle{remark}
\newtheorem{remark}{Remark}[section]
\newtheorem{notation}{Notations}[section]
\newtheorem{example}{Example}[section]

\renewcommand{\geq}{\geqslant}
\renewcommand{\leq}{\leqslant}
\newcommand{\RF}{\mathcal{RF}}
\newcommand{\K}{\mathbb{K}}
\newcommand{\tdelta}{\tilde{\Delta}}
\newcommand{\T}{\mathcal{T}}

\newcommand{\D}{\mathcal{D}}
\newcommand{\PT}{\mathcal{PT}}
\newcommand{\UPT}{\mathcal{UPT}}
\newcommand{\g}{\mathfrak{g}}
\renewcommand{\PT}{\mathcal{PT}}
\newcommand{\N}{\mathbb{N}}
\newcommand{\h}{\mathcal{H}}
\newcommand{\Id}{\mathcal{I}d}
\newcommand{\id}{\mathrm{Id}}
\newcommand{\prim}{\mathrm{Prim}}
\newcommand{\vect}{\mathrm{Vect}}
\renewcommand{\ker}{\mathrm{Ker}}
\newcommand{\UCP}{\mathit{UCP}}
\newcommand{\CP}{\mathit{CP}}
\newcommand{\Sh}{\mathit{Sh}}
\newcommand{\sh}{\mathcal{S}h}
\newcommand{\CK}{\mathit{CK}}
\newcommand{\bl}{\mathcal{B}l}
\newcommand{\st}{\mathcal{S}t}
\newcommand{\morp}{\mathrm{End}}

\begin{document}

\maketitle

\begin{abstract}
A Com-PreLie bialgebra is a commutative bialgebra with an extra preLie product satisfying some compatibilities with the product and the coproduct.
We here give examples of cofree Com-PreLie bialgebras, including all the ones such that the preLie product is homogeneous of degree $\geq -1$. 
We also give a graphical description of free unitary Com-PreLie algebras, explicit their canonical bialgebra structure
and exhibit with the help of a rigidity theorem certain cofree quotients, including the Connes-Kreimer Hopf algebra of rooted trees.
We finally prove that the dual of these bialgebras are also enveloping algebras of preLie algebras, combinatorially described.
\end{abstract}


\tableofcontents

\section*{Introduction}

Com-PreLie bialgebras, introduced in \cite{Foissy28,Foissy30}, are commutative bialgebras with an extra preLie product, compatible
with the product and coproduct, see Definition \ref{defi1} below. They appeared in Control Theory, as the Lie algebra of the group of Fliess operators
\cite{Gray2011} naturally owns a Com-PreLie bialgebra structure, and its underlying bialgebra is a shuffle Hopf algebra.
Free (non unitary) Com-PreLie bialgebras were also described, in terms of partitioned rooted trees.

We here give examples of connected cofree Com-PreLie bialgebras. As cocommutative cofree bialgebras are, up to isomorphism, shuffle algebras
$\Sh(V)=(T(V),\shuffle,\Delta)$, where $V$ is the space of primitive elements, we firstly characterize Com-PreLie bialgebras structures on $\Sh(V)$
in term of operators $\varpi:T(V)\otimes T(V)\longrightarrow V$, satisfying two identities, see Proposition \ref{prop9}. In particular,
if we assume that the obtained preLie bracket is homogeneous of degree $0$ for the graduation of $\Sh(V)$ by the length, then $\varpi$ is reduced to a linear map $f:V\longrightarrow V$,
and the obtained preLie product is given by (Proposition \ref{propf}):
\begin{align*}
&\forall x_1,\ldots,x_m,y_1,\ldots,y_n\in V,&
x_1\ldots x_m\bullet y_1\ldots y_n=\sum_{i=0}^n x_1\ldots x_{i-1}f(x_i)(x_{i+1}\ldots x_m\shuffle y_1\ldots y_n).
\end{align*}
In particular, if $V=\vect(x_0,x_1)$ and $f$ is defined by $f(x_0)=0$ and $f(x_1)=x_0$, we obtain the Com-PreLie bialgebra of Fliess operators
in dimension $1$. If we assume that the obtained preLie bracket is homogeneous of degree $-1$, then $\varpi$ is given by two bilinear products
$*$ and $\{-,-\}$ on $V$ such that $*$ is preLie, $\{-,-\}$ is antisymmetric and for all $x,y,z\in V$,
\begin{align*}
x*\{y,z\}&=\{x*y,z\},\\
\{x,y\}*z&=\{x*y,z\}+\{x,y*z\}+\{\{x,y\},z\}.
\end{align*}
This includes preLie products on $V$ when $\{-,-\}=0$ and nilpotent Lie algebras of nilpotency order $2$ when $*=0$, see Proposition \ref{prop11}.\\

We then extend the construction of free Com-PreLie algebras of \cite{Foissy28} in terms of partitioned trees (see Definition \ref{defiarbres})
to free unitary Com-PreLie algebras $\UCP(\D)$, with the help of a complementary decoration by integers. 
We obtain free Com-PreLie algebras $\CP(\D)$ as the augmentation ideal of a quotient of $\UCP(\D)$, the right action of the unit $\emptyset$
on the generators of $\UCP(\D)$ being arbitrarily chosen (Proposition \ref{prop16}). 
Recall that partitioned trees are rooted forests with an extra structure of a partition of its vertices into blocks;
forgetting the blocks, we obtain the Connes-Kreimer Hopf  algebra $\h_{\CK}$ of rooted trees \cite{Connes98},
which is given in this way a natural structure of Com-PreLie bialgebra (Proposition \ref{prop17}).
Using Livernet's rigidity theorem for preLie algebras, we prove that the augmentation ideals of $\UCP(\D)$ and $\CP(\D)$ are free as preLie algebras.
Theorem \ref{theo27} is a rigidity theorem which gives a simple criterion for a connected (as a coalgebra) Com-PreLie bialgebra to be cofree,
in terms of the right action of the unit on its primitive elements. Applied to $\CP(\D)$ and $\h_{\CK}$, it proves that they are isomorphic to shuffle
bialgebras, which was already known for $\h_{\CK}$. We also consider the dual Hopf algebras of $\UCP(\D)$ and $\CP(\D)$:
as these Hopf algebras are right-sided combinatorial in the sense of \cite{Loday2010}, there dual are enveloping algebras of other preLie algebras,
which we explicitly describe in Theorem \ref{theo31}, and then compare to the original Com-PreLie algebras.\\

This text is organized as follows: the first section contains reminders and lemmas on Com-PreLie algebras,
including the extension of the Guin-Oudom extension of the preLie product in the Com-PreLie case.
The second section deals with the characterization of preLie products on shuffle algebras.
In the next section contains the description of free unitary Com-PreLie algebras and two families of quotients,
whereas the fifth and last one contains results on the bialgebraic structures of these objects: existence of the coproduct,
the rigidity theorem \ref{theo27} and its applications, the dual preLie algebras, and an application to a family of subalgebras,
named Connes-Moscovici subalgebras. \\ 

\textbf{Acknowledgments}. The author thanks the anonymous referee for his helpful comments. 
The author acknowledges support from the grant ANR-20-CE40-0007 \emph{Combinatoire Algébrique, Résurgence, Probabilités Libres et Opérades}.\\

\begin{notation} \begin{enumerate}
\item Let $\K$ be a commutative field of characteristic zero. All the objects (vector spaces, algebras, coalgebras, preLie algebras$\ldots$)
in this text will be taken over $\K$.
\item For all $n\in \mathbb{N}$, we denote by $[n]$ the set $\{1,\ldots,n\}$. In particular, $[0]=\emptyset$.
\end{enumerate}\end{notation}

 \section{Reminders on Com-PreLie algebras}

Let $V$ be a vector space.
\begin{itemize}
\item We denote by $T(V)$ the tensor algebra of $V$. Its unit is the empty word, which we denote by $\emptyset$.
The element $v_1\otimes \ldots \otimes v_n \in V^{\otimes n}$, with $v_1,\ldots,v_n \in V$, will be shortly denoted by $v_1\ldots v_n$.
The deconcatenation coproduct of $T(V)$ is defined by
\begin{align*}
&\forall v_1,\ldots,v_n \in V,&\Delta(v_1\ldots v_n)&=\sum_{i=0}^n v_1\ldots v_i \otimes v_{i+1}\ldots v_n.
\end{align*}
The shuffle product of $T(V)$ is denoted by $\shuffle$. 
Recall that it can be inductively defined by
\begin{align*}
&\forall  x,y \in V,\: \forall u,v \in T(V),&\emptyset \shuffle v&=0,&xu\shuffle yv&=x(u\shuffle yv)+y(xu\shuffle v).
\end{align*}
For example, if $v_1,v_2,v_3,v_4\in V$,
\begin{align*}
v_1\shuffle v_2v_3v_4&=v_1v_2v_3v_4+v_2v_1v_3v_4+v_2v_3v_1v_4+v_2v_3v_4v_1,\\
v_1v_2\shuffle v_3v_4&=v_1v_2v_3v_4+v_1v_3v_2v_4+v_1v_3v_4v_2+v_3v_1v_2v_4+v_3v_1v_4v_2+v_3v_4v_1v_2,\\
v_1v_2v_3\shuffle v_4&=v_1v_2v_3v_4+v_1v_2v_4v_3+v_1v_2v_4v_3+v_1v_4v_2v_3+v_4v_1v_2v_3.
\end{align*}
$\Sh(V)=(T(V),\shuffle,\Delta)$ is a Hopf algebra, known as the shuffle algebra of $V$.
\item $S(V)$ is the symmetric algebra of $V$. It is a Hopf algebra, with the coproduct defined by
\begin{align*}
&\forall v\in V,&\Delta(v)&=v\otimes \emptyset+\emptyset\otimes v.
\end{align*}
\item $coS(V)$ is the subalgebra of $(T(V),\shuffle)$ generated by $V$. It is the greatest cocommutative Hopf subalgebra of $(T(V),\shuffle,\Delta)$,
and is isomorphic to $S(V)$ via the algebra morphism
\[\theta:\left\{\begin{array}{rcl}
(S(V),m,\Delta)&\longrightarrow&(coS(V),\shuffle,\Delta)\\
v_1\ldots v_k&\longrightarrow&v_1\shuffle \ldots \shuffle v_k.\\
\end{array}\right.\]
\end{itemize}

\subsection{Definitions}

\begin{defi}\label{defi1}
\begin{enumerate}
\item A \emph{Com-PreLie algebra} \cite{Foissy28,Foissy30}  is a family $A=(A,\cdot,\bullet)$, where $A$ is a vector space,
$\cdot$ and $\bullet$ are bilinear products on $A$, such that
\begin{align*}
&\forall a,b\in A,&a\cdot b&=b\cdot a,\\
&\forall a,b,c\in A,&(a\cdot b)\cdot c&=a\cdot (b\cdot c),\\
&\forall a,b,c\in A,&(a\bullet b)\bullet c-a\bullet(b\bullet c)&=(a\bullet c)\bullet b-a\bullet(c\bullet b)&\mbox{(preLie identity)},\\
&\forall a,b,c\in A,&(a\cdot b)\bullet c&=(a\bullet c)\cdot b+a\cdot (b\bullet c)&\mbox{(Leibniz identity)}.
\end{align*}
In particular, $(A,\cdot)$ is an associative, commutative algebra and $(A,\bullet)$ is a right preLie algebra.
We shall say that $A$ is unitary if the algebra $(A,\cdot)$ is unitary.
\item A \emph{Com-PreLie bialgebra} is a family $(A,\cdot,\bullet,\Delta)$, such that
\begin{enumerate}
\item $(A,\cdot,\bullet)$ is a unitary Com-PreLie algebra.
\item $(A,\cdot,\Delta)$ is a bialgebra.
\item For all $a,b\in A$,
\[\Delta(a\bullet b)=a^{(1)}\otimes a^{(2)}\bullet b+a^{(1)}\bullet b^{(1)}\otimes a^{(2)}\cdot b^{(2)},\]
with Sweedler's notation $\Delta(x)=x^{(1)}\otimes x^{(2)}$.
\end{enumerate}
\end{enumerate}\end{defi}

\begin{remark} If $(A,\cdot,\bullet,\Delta)$ is a Com-PreLie bialgebra, then for any $\lambda \in \K$, $(A,\cdot,\lambda \bullet,\Delta)$ also is. 
\end{remark}

\begin{lemma}\label{lem3} \begin{enumerate}
\item Let $(A,\cdot,\bullet)$ be a unitary Com-PreLie algebra. Its unit is denoted by $\emptyset$.
For all $a \in A$, $\emptyset\bullet a=0$.
\item  Let $A$ be a Com-PreLie bialgebra, with counit $\varepsilon$. For all $a,b \in A$, $\varepsilon(a\bullet b)=0$.
\end{enumerate}\end{lemma}

\begin{proof} 1. Indeed, $\emptyset\bullet a=(\emptyset\cdot\emptyset)\bullet a=(\emptyset\bullet a)\cdot\emptyset
+\emptyset\cdot(\emptyset\bullet a)=2(\emptyset\bullet a)$, so $\emptyset\bullet a=0$.\\

2. For all $a,b \in A$,
\begin{align*}
\varepsilon(a\bullet b)&=(\varepsilon \otimes \varepsilon)\circ \Delta(a\bullet b)\\
&=\varepsilon(a^{(1)})\varepsilon( a^{(2)}\bullet b)+\varepsilon(a^{(1)}\bullet b^{(1)})\varepsilon(a^{(2)}
\cdot b^{(2)})\\
&=\varepsilon(a^{(1)})\varepsilon( a^{(2)}\bullet b)
+\varepsilon(a^{(1)}\bullet b^{(1)})\varepsilon(a^{(2)})\varepsilon(b^{(2)})\\
&=\varepsilon(a\bullet b)+\varepsilon(a\bullet b),
\end{align*}
so $\varepsilon(a\bullet b)=0$. \end{proof}

\begin{remark} 
Let us give a few reminders on the (dual) Hochschild cohomology for coalgebras, also called Cartier-Quillen cohomology, see \cite{Connes98}.
Let $(C,\Delta)$ be a coalgebra, and $(M,\rho_L,\rho_R)$ be a bicomodule over $C$, with left coaction $\rho_L$
and right coaction $\rho_R$. A $n$-cochain is a map $L:M\longrightarrow C^{\otimes n}$. The coboundary $d$ is given on any $n$-cochain $n$ by
\[d(L)=(\id \otimes L)\circ \rho_L+\sum_{i=1}^n (\id^{\otimes (i-1)}\otimes \Delta\otimes \id^{\otimes (n-i)})\circ L+(-1)^{n-1}(L\otimes \id)\circ \rho_R.\]
In particular, if $(B,m,\Delta)$ is a bialgebra, we can consider the bicomodule $(M,\rho_L,\rho_R)$ defined by $M=B$, $\rho_R=\Delta$ and
\begin{align*}
&\forall x\in B,&\rho_L(x)&=1\otimes x.
\end{align*}
A 1-cocycle is then a map $L:B\longrightarrow B$ such that for any $x\in B$,
\[\Delta\circ L(x)=1\otimes L(x)+(L\otimes \id)\circ \Delta.\]
Observe that in any Com-PreLie bialgebra $A$, if $a$ is primitive, for any $b\in A$,
\begin{align}
\label{eqcocycle}
\Delta(a\bullet b)=\emptyset\otimes a\bullet b+a\bullet b^{(1)}\otimes b^{(2)}.
\end{align}
Therefore, the map $b\mapsto a\bullet b$ is a $1$-cocycle for this cohomology.
\end{remark}

\subsection{Linear endomorphism on primitive elements}

\begin{notation}
If $A$ is a bialgebra, we denote by $\prim(A)$ the space of its primitive elements. 
\end{notation}

\begin{prop} \label{prop4}
Let $A$ be a Com-PreLie bialgebra. Its unit is denoted by $\emptyset$. 
\begin{enumerate}
\item If $x \in \prim(A)$, then $x\bullet \emptyset\in \prim(A)$.
We denote by $f_A$ the map 
\[f_A:\left\{\begin{array}{rcl}
\prim(A)&\longrightarrow&\prim(A)\\
a&\longrightarrow&a\bullet \emptyset.
\end{array}\right.\]
\item $\prim(A)$ is a preLie subalgebra of $(A,\bullet)$ if, and only if, $f_A=0$.
\end{enumerate} \end{prop}

\begin{proof} 1. Indeed, if $a$ is primitive,
\[\Delta(a\bullet \emptyset)=a\otimes \emptyset\bullet \emptyset+\emptyset\otimes a\bullet \emptyset
+a\bullet \emptyset\otimes \emptyset\cdot \emptyset+\emptyset\bullet \emptyset\otimes a\cdot \emptyset
=0+\emptyset\otimes \emptyset\bullet a+a\bullet \emptyset\otimes \emptyset+0,\]
so $a\bullet \emptyset$ is primitive.\\

2. Let $a,b \in \prim(A)$.
\begin{align*}
\Delta(a\bullet b)&=a\otimes \emptyset \bullet b+\emptyset\otimes a\bullet b+\emptyset\bullet \emptyset\otimes a\cdot b
+a\bullet \emptyset\otimes b+\emptyset\bullet b\otimes a+a\bullet b\otimes \emptyset\\
&=\emptyset\otimes a\bullet b+a\bullet b\otimes \emptyset+f_A(a)\otimes b. 
\end{align*}
Hence, $\prim(A)$ is a preLie subalgebra if, and only if, for any $a,b\in A$, $f_A(a)\otimes b=0$, 
that is to say if, and only if, $f_A=0$. \end{proof}

\subsection{Extension of the pre-Lie product}

Let $A$ be a Com-PreLie algebra. It is a Lie algebra, with the bracket defined by 
\begin{align*}
&\forall x,y\in A,& [x,y]&=x\bullet y-y\bullet x.
\end{align*}
We shall use the Oudom-Guin construction of its enveloping algebra \cite{Oudom2008,Oudom2005}. 
In order to avoid confusions, we shall denote by $\times$ the usual product of $S(A)$ and by $1$ its unit. 
We extend the preLie product $\bullet$ into a product from $S(A) \otimes S(A)$ into $S(A)$ by
\begin{itemize}
\item If $a_1,\ldots,a_k \in A$, $(a_1\times \ldots \times a_k) \bullet 1=a_1\times \ldots \times a_k$.
\item If $a,a_1,\ldots,a_k \in A$,
\[a\bullet (a_1\times\ldots \times a_k)=(a\bullet (a_1\times \ldots \times a_{k-1}))\bullet a_k
-\sum_{i=1}^{k-1}a\bullet (a_1\times \ldots\times (a_i \bullet a_k) \times\ldots \times a_{k-1}).\]
\item If $x,y,z\in S(A)$, $(x\times y)\bullet z=(x\bullet z^{(1)})\times (y\bullet z^{(2)})$,
where $\Delta(z)=z^{(1)}\otimes z^{(2)}$ is the usual coproduct of $S(A)$.
%\[\Delta(u_1\times \ldots \times u_k)=(u_1\otimes 1+1\otimes u_1)\times \ldots \times (u_k \otimes 1+1\otimes u_k),\]
%for all $u_1,\ldots, u_k \in A$.
\end{itemize}

\begin{notation} If $c_1,\ldots,c_n \in A$ and $I=\{i_1,\ldots,i_k\}\subseteq [n]$, we put
\[\prod_{i\in I}^\times c_i=c_{i_1}\times \ldots \times c_{i_k}.\]
\end{notation}

\begin{prop} \label{prop6}
 \begin{enumerate}
\item Let $A$ be a Com-PreLie algebra. If $a,b,c_1,\ldots,c_n \in A$,
\[(a\cdot b)\bullet (c_1\times \ldots \times c_k)=\sum_{I\subseteq [n]} \left(a\bullet \prod_{i\in I}^\times c_i\right) \cdot
\left(b\bullet \prod_{i\notin I}^\times c_i\right).\]
\item Let $A$ be a Com-PreLie bialgebra. If $a,b_1,\ldots, b_n \in A$,
\[\Delta(a\bullet (b_1\times \ldots \times b_n))=\sum_{I\subseteq [n]} a^{(1)}\bullet\left(\prod_{i\in I}^\times b_i^{(1)}\right)\otimes 
\left(\prod_{i\in I} b_i^{(2)}\right)a^{(2)}\bullet \left(\prod_{i\notin I}^\times b_i\right).\]
\end{enumerate}\end{prop}

\begin{proof}  These are proved by direct, but quite long, inductions on $n$. \end{proof}

\begin{lemma}\label{lemmacoproduit} 
Let $A$ be a Com-PreLie bialgebra. For all $a\in \prim(A)$, $k\geq 0$, $b_1,\ldots, b_l \in A$,
\[a\bullet \emptyset^{\times k}\times b_1\times \ldots \times b_l=f_A^k(a)\bullet b_1\times \ldots \times b_l.\]
\end{lemma}

\begin{proof} This is obvious if $k=0$. Let us prove it for $k=1$ by induction on $l$. It is obvious if $l=0$.
Let us assume the result at rank $l-1$. Then 
\begin{align*}
a\bullet \emptyset \times b_1\times \ldots \times b_l&=(a\bullet \emptyset \times b_1\times \ldots \times b_{l-1})\bullet b_l
+a\bullet (\emptyset \bullet b_l)\times b_1\times \ldots \times b_{l-1}\\
&+\sum_{i=1}^{l-1}a\bullet \emptyset \times b_1\times \ldots \times (b_i \bullet b_l)\times \ldots \times b_{l-1}\\
&=(f_A(a)\bullet b_1\times \ldots \times b_{l-1})\bullet b_l+0+\sum_{i=1}^{l-1}f_A(a)\bullet b_1\times \ldots \times (b_i \bullet b_l)\times \ldots \times b_{l-1}\\
&=f_A(a)\bullet b_1\times \ldots \times b_l.
\end{align*}
The result is proved for $k\geq 2$ by an induction on $k$. \end{proof}

\section{Examples on shuffle algebras}

\subsection{Preliminary lemmas}

We shall denote by $\pi:T(V)\longrightarrow V$ the canonical projection.

\begin{lemma} \label{lemma1}
Let $\varpi:T(V)\otimes T(V)\longrightarrow V$ be a linear map. 
\begin{enumerate}
\item There exists a unique map $\bullet:T(V)\otimes T(V)\longrightarrow T(V)$ such that 
\begin{enumerate}
\item $\pi \circ \bullet=\varpi$.
\item For all $u,v\in T(V)$,
\begin{align}
\label{EQ1}
\Delta(u\bullet v)&=u^{(1)}\otimes u^{(2)}\bullet v+u^{(1)}\bullet v^{(1)}\otimes u^{(2)}\shuffle v^{(2)}.
\end{align}
\end{enumerate}
This product $\bullet$ is given by 
\begin{align}
\label{eqbullet}
&\forall u,v\in T(V),&u\bullet v&=u^{(1)}\varpi(u^{(2)}\otimes v^{(1)})(u^{(3)}\shuffle v^{(2)}).
\end{align}
\item The following conditions are equivalent:
\begin{enumerate}
\item For all $u,v,w\in T(V)$,
\begin{align*}
(u\shuffle v)\bullet w&=(u\bullet w)\shuffle v+u\shuffle (v\bullet w).
\end{align*}
\item For all $u,v,w\in T(V)$,
\begin{align}
\label{EQ2} \varpi((u\shuffle v)\otimes w)&=\varepsilon(u) \varpi(v\otimes w)+\varepsilon(v)\varpi(u\otimes w).
\end{align}
\end{enumerate}
\item Let $N\in \mathbb{Z}$. The following conditions are equivalent:
\begin{enumerate}
\item $\bullet$ is homogeneous of degree $N$, that is to say 
\begin{align*}
&\forall k,l\geq 0,&V^{\otimes k}\bullet V^{\otimes l}\subseteq V^{\otimes (k+l+N)}.
\end{align*}
\item For all $k,l\geq 0$, such that $k+l+N\neq 1$, $\varpi(V^{\otimes k}\otimes V^{\otimes l})=(0)$.
\end{enumerate}
We use the convention $V^{\otimes p}=(0)$ if $p<0$. \end{enumerate}\end{lemma}

\begin{proof} 1. \textit{Existence.} Let $\bullet$ be the product on $T(V)$ defined by 
\begin{align*}
&\forall u,v\in T(V),&u\bullet v&=u^{(1)}\varpi(u^{(2)}\otimes v^{(1)})(u^{(3)}\shuffle v^{(2)}).
\end{align*}
As $\varpi$ takes its values in $V$, for all $u,v\in T(V)$,
\begin{align*}
\pi(u\bullet v)&=\varepsilon(u^{(1)}) \varpi(u^{(2)}\otimes v^{(1)}) \varepsilon(u^{(3)}\shuffle v^{(2)})\\
&=\varepsilon(u^{(1)}) \varpi(u^{(2)}\otimes v^{(1)}) \varepsilon(u^{(3)})\varepsilon(v^{(2)})\\
&=\varpi(u\otimes v).
\end{align*}
We denote by $m$ the concatenation product of $T(V)$. As $(T(V),m,\Delta)$ is an infinitesimal bialgebra (see \cite{Loday2004-2,Loday2006}),
for all $u,v\in T(V)$,
\begin{align*}
\Delta(u\bullet v)&=u^{(1)}\otimes u^{(2)} \varpi(u^{(3)}\otimes v^{(1)})(u^{(4)}\shuffle v^{(2)})
+u^{(1)}\varpi(u^{(2)}\otimes v^{(1)})\otimes u^{(3)}\shuffle v^{(2)}\\
&+u^{(1)}\otimes \varpi(u^{(2)}\otimes v^{(1)})(u^{(3)}\shuffle v^{(2)})+u^{(1)}\varpi(u^{(2)}\otimes v^{(1)})(u^{(3)}\shuffle v^{(2)})\otimes u^{(4)}\shuffle v^{(3)}\\
&-u^{(1)}\varpi(u^{(2)}\otimes v^{(1)})\otimes u^{(3)}\shuffle v^{(2)}-u^{(1)}\otimes \varpi(u^{(2)}\otimes v^{(1)})(u^{(3)}\otimes v^{(2)})\\
&=u^{(1)}\otimes u^{(2)} \varpi(u^{(3)}\otimes v^{(1)})(u^{(4)}\shuffle v^{(2)})
+u^{(1)}\varpi(u^{(2)}\otimes v^{(1)})(u^{(3)}\shuffle v^{(2)})\otimes u^{(4)}\shuffle v^{(3)}\\
&=u^{(1)}\otimes u^{(2)}\bullet v+u^{(1)}\bullet v^{(1)}\otimes u^{(2)}\shuffle v^{(2)}.
\end{align*}

\textit{Unicity}. Let $\diamond$ be another product satisfying the required properties. Let us denote that $u\diamond v=u\bullet v$ for any words 
$u,v$ of respective lengths $k$ and $l$. If $k=0$, then we can assume that $u=\emptyset$.  We proceed by induction on $l$. 
If $l=0$, then we can assume that $v=\emptyset$. By (\ref{EQ1}), $\emptyset\bullet \emptyset$ and 
$\emptyset \diamond \emptyset$ are primitive elements of $T(V)$, so belong to $V$. Hence,
\[\emptyset\bullet \emptyset=\pi(\emptyset\bullet \emptyset)=\varpi(\emptyset\otimes \emptyset)
=\pi(\emptyset\diamond \emptyset)=\emptyset\diamond\emptyset.\]
If $l\geq 1$, then, by (\ref{EQ1}),
\begin{align*}
\Delta(\emptyset\bullet v)&=\emptyset\otimes \emptyset\bullet v+\emptyset\bullet v\otimes \emptyset+\emptyset\bullet \emptyset\otimes v
+\emptyset\bullet v^{(1)}\otimes v^{(2)},\\
\tdelta(\emptyset\bullet v)&=\emptyset\bullet \emptyset\otimes v+\emptyset\bullet v^{(1)}\otimes v^{(2)}.
\end{align*}
The same computation for $\diamond$ and the induction hypothesis on $l$, applied to $(\emptyset,v^{(1)})$, imply that
$\tdelta(\emptyset\bullet v-\emptyset\diamond v)=0$, so $\emptyset\bullet v-\emptyset\diamond v\in V$. Finally,
\begin{align*}
\emptyset\bullet v-\emptyset\diamond v&=\pi(\emptyset\bullet v-\emptyset\diamond v)=\varpi(\emptyset\otimes v-\emptyset\otimes v)=0.
\end{align*}
If $k\geq 1$, we proceed by induction on $l$. If $l=0$, we can assume that $v=\emptyset$; 
(\ref{EQ1}) implies that $\tdelta(u\bullet \emptyset-u\diamond \emptyset)=0$,
so $u\bullet \emptyset-u\diamond \emptyset=0$ and, applying $\pi$, finally $u\bullet\emptyset=u\diamond \emptyset$.
If $l\geq 1$, by (\ref{EQ1}), the induction hypothesis on $k$ applied to $(u^{(1)},v)$ and the induction hypothesis on $l$
applied to $(u,\emptyset)$ and $(u,v^{(1)})$ gives
\begin{align*}
\tdelta(u\bullet v)&=u^{(1)}\otimes u^{(2)}\bullet v+u\bullet \emptyset\otimes v+u\bullet v^{(1)}\otimes v^{(2)}\\
&=u^{(1)}\otimes u^{(2)}\diamond v+u\diamond \emptyset\otimes v+u\diamond v^{(1)}\otimes v^{(2)}=\tdelta(u\diamond v).
\end{align*}
As before, $u\bullet v=u\diamond v$.\\

2. $\Longrightarrow$. As $\varpi$ takes its values in $V$, we have 
 \begin{align*}
 \varpi(u\shuffle v)\otimes w)&=\varpi ((u\bullet w)\shuffle v+u\shuffle (v\bullet w))\\
 &=\varepsilon(v)\varpi(u\otimes w)+\varepsilon(u)\varpi(v\otimes w).
 \end{align*}
 $\Longleftarrow$. For all $u,v,w\in T(V)$,
\begin{align*}
(u\shuffle v)\bullet w&=(u^{(1)}\shuffle v^{(1)})\varpi((u^{(2)}\shuffle v^{(2)})\otimes w^{(1)})(u^{(3)}\shuffle v^{(3)}\shuffle w^{(2)})\\
&=\varepsilon(u^{(2)})(u^{(1)}\shuffle v^{(1)})\varpi(v^{(2)}\otimes w^{(1)})(u^{(3)}\shuffle v^{(3)}\shuffle w^{(2)})\\
&+\varepsilon(v^{(2)})(u^{(1)}\shuffle v^{(1)})\varpi(u^{(2)}\otimes w^{(1)})(u^{(3)}\shuffle v^{(3)}\shuffle w^{(2)})\\
&=(u^{(1)}\shuffle v^{(1)})\varpi(v^{(2)}\otimes w^{(1)})(u^{(2)}\shuffle v^{(3)}\shuffle w^{(2)})\\
&+(u^{(1)}\shuffle v^{(1)})\varpi(u^{(2)}\otimes w^{(1)})(u^{(3)}\shuffle v^{(2)}\shuffle w^{(2)})\\
&=u\shuffle \left(v^{(1)}\varpi(v^{(2)}\otimes w^{(1)})(v^{(3)}\shuffle w^{(2)})\right)\\
&+v\shuffle \left(u^{(1)}\varpi(u^{(2)}\otimes w^{(1)})(u^{(3)}\shuffle w^{(2)})\right)\\
&=u\shuffle (v\bullet w)+(u\bullet w)\shuffle v.
\end{align*}
So the compatibility between $\shuffle$ and $\bullet$ is satisfied. \\

3. $(a)\Longrightarrow (b)$: immediately implied by $\varpi=\pi\circ \bullet$.
$(b)\Longrightarrow (a)$: comes from (\ref{eqbullet}).
 \end{proof}

\begin{remark} If (\ref{EQ2}) is satisfied, for $u=v=\emptyset$, we obtain 
\begin{align*}
&\forall w\in T(V),&\varpi(\emptyset \otimes w)=0.
\end{align*}\end{remark}

\begin{lemma}
Let $\varpi:T(V)\otimes T(V)\longrightarrow V$, satisfying (\ref{EQ2}), and let $\bullet$ be the product associated to $\varpi$ in Lemma \ref{lemma1}.
Then $(T(V),\shuffle,\bullet,\Delta)$ is a Com-PreLie bialgebra if, and only if 
\begin{align}\label{EQ3}
&\forall u,v,w\in T(V),&\varpi(u\bullet v\otimes w)-\varpi(u\otimes v\bullet w)=
\varpi(u\bullet w\otimes v)-\varpi(u\otimes w\bullet v).
\end{align} \end{lemma}

\begin{proof} $\Longrightarrow$. This is immediately obtained by applying $\pi$ to the preLie identity, as $\varpi=\pi\circ \bullet$.\\

$\Longleftarrow$. By lemma \ref{lemma1}, it remains to prove that $\bullet$ is preLie.
For any $u,v,w\in T(V)$, we put
\[PL(u,v,w)=(u\bullet v)\bullet w-u\bullet(v\bullet w)-(u\bullet w)\bullet v+u\bullet(w\bullet v).\]
By hypothesis, $\pi\circ PL(u,v,w)=0$ for any $u,v,w\in T(V)$.
Let us prove that $PL(u,v,w)=0$ for any $u,v,w\in T(V)$.
A direct computation using (\ref{EQ1}) shows that
\begin{align}\label{EQ4}
\Delta(PL(u,v,w))&=u^{(1)}\otimes PL(u^{(2)},v,w)\otimes u^{(1)}+PL(u^{(1)},v^{(1)},w^{(1)})\otimes u^{(2)}\shuffle v^{(2)}\shuffle w^{(2)}.
\end{align}
Let $v\in T(V)$. Then
\begin{align*}
\emptyset\bullet v&=(\emptyset\shuffle \emptyset)\bullet v=(\emptyset\bullet v)\shuffle \emptyset+\emptyset\shuffle(\emptyset\bullet v)=2 \emptyset\bullet v,
\end{align*}
so $\emptyset \bullet v=0$ for any $v\in T(V)$. Hence, for any $v,w\in T(V)$, $PL(\emptyset,v,w)=0$: by trilinearity of $PL$, we can assume that
$\varepsilon(u)=0$. In this case, (\ref{EQ4}) becomes 
\begin{align*}
\Delta(PL(u,v,w))&=\emptyset \otimes PL(u,v,w)+PL(u,v^{(1)},w^{(1)})\otimes v^{(2)}\shuffle w^{(2)}\\
&+PL(u^{(1)},v^{(1)},w^{(1)})\otimes u^{(2)}\shuffle v^{(2)}\shuffle w^{(2)}.
\end{align*}
We assume that $u,v,w$ are words of respective lengths $k$, $l$ and $n$, with $k\geq 1$.
Let us first prove that $PL(u,v,w)=0$ if $l=0$, or equivalently if $v=\emptyset$, by induction on $n$.
 If $n=0$, then we can take $w=\emptyset$
and, obviously, $PL(u,\emptyset,\emptyset)=0$. If $n\geq 1$, (\ref{EQ4}) becomes 
\begin{align*}
\Delta(PL(u,\emptyset,w))&=\emptyset \otimes PL(u,v,w)+PL(u,\emptyset,w^{(1)})\otimes w^{(2)}\\
&=\emptyset \otimes PL(u,v,w)+PL(u,\emptyset,w)\otimes \emptyset+PL(u,\emptyset,w^{(1)})\otimes w^{(2)}.
\end{align*}
By the induction hypothesis on $n$, $PL(u,\emptyset,w^{(1)})=0$, so $PL(u,\emptyset,w)$ is primitive, so belongs to $V$.
As $\pi\circ PL=0$, $PL(u,\emptyset,w)=0$. 

Therefore, we can now assume that $l\geq 1$. By symmetry in $v$ and $w$, we can also assume that $n\geq 1$. Let us now prove 
that $PL(u,v,w)=0$ by induction on $k$. If $k=0$, there is nothing more to prove. If $k\geq 1$, we proceed by induction on $l+n$. 
If $l+n\leq 1$, there is nothing more to prove. Otherwise, using both induction hypotheses, (\ref{EQ4}) becomes 
\begin{align*}
\Delta(PL(u,v,w))&=PL(u,v,w)\otimes \emptyset+\emptyset\otimes PL(u,v,w).
\end{align*}
So $PL(u,v,w)\in V$. As $\pi\circ PL=0$, $PL(u,v,w)=0$. \end{proof}

Consequently: 

\begin{prop} \label{prop9}
Let $\varpi:T(V)\otimes T(V) \longrightarrow V$ be a linear map such that (\ref{EQ2}) and (\ref{EQ3}) are satisfied.
The product $\bullet$ defined by (\ref{EQ1}) makes $(T(V),\shuffle,\bullet,\Delta)$ a Com-PreLie bialgebra.
We obtain in this way all the preLie products $\bullet$ such that $(T(V),\shuffle,\bullet,\Delta)$ a Com-PreLie bialgebra.
Moreover, for any $N\in \mathbb{Z}$, $\bullet$ is homogeneous of degree $N$ if, and only if 
\begin{align}\label{EQ5}
&\forall k,l\in \mathbb{N},& k+l+N\neq 1&\Longrightarrow \varpi(V^{\otimes k}\otimes V^{\otimes l})=(0).
\end{align}\end{prop}

\begin{remark} Let $\varpi:T(V)\otimes T(V)\longrightarrow V$, satisfying (\ref{EQ5}) for a given $N\in \mathbb{Z}$. Then 
\begin{enumerate}
\item  (\ref{EQ2}) is satisfied if, and only if, for all $k,l,n\in \mathbb{N}$ such that $k+l+n=1-N$,
\begin{align*}
&\forall u\in V^{\otimes k},\:\forall v\in V^{\otimes l},\: \forall w\in V^{\otimes n},
&\varpi((u\shuffle v)\otimes w)&=\varepsilon(u)\varpi(v\otimes w)+\varepsilon(v)\varpi(u\otimes w).
\end{align*}
\item (\ref{EQ3}) is satisfied if, and only if, for all $k,l,n\in \mathbb{N}$ such that $k+l+n=1-2N$,
\begin{align*}
&\forall \in V^{\otimes k},\:\forall v\in V^{\otimes l},\:\forall w\in V^{\otimes n},
&\varpi(u\bullet v\otimes w)-\varpi(u\otimes v\bullet w)\\
&&=\varpi(u\bullet w\otimes v)-\varpi(u\otimes w\bullet v).
\end{align*} \end{enumerate}
Note that (\ref{EQ2}) is always satisfied if $u=\emptyset$ or $v=\emptyset$, that is to say if $k=0$ or $l=0$.
\end{remark}

In the next paragraphs, we shall look at $N\geq 0$ and $N=-1$.

\subsection{PreLie products of positive degree}

\begin{prop}\label{propf}
Let $f$ be a linear endomorphism of $V$. 
We define a product $\bullet$ on $T(V)$ by
\begin{align}\label{EQ6}
&\forall x_1,\ldots,x_m,y_1,\ldots,y_n\in V,&
x_1\ldots x_m\bullet y_1\ldots y_n=\sum_{i=0}^n x_1\ldots x_{i-1}f(x_i)(x_{i+1}\ldots x_m\shuffle y_1\ldots y_n).
\end{align}
Then $(T(V),\shuffle,\bullet,\Delta)$ is a Com-PreLie bialgebra denoted by $T(V,f)$. 
Conversely, if $\bullet$ is a product on $T(V)$, homogeneous of degree $N\geq 0$, there exists a unique $f:V\longrightarrow V$
such that $(T(V),\shuffle,\bullet,\Delta)=T(V,f)$.
\end{prop}

\begin{proof} We look for all possible $\varpi$, homogeneous of a certain degree $N\geq 0$, such that (\ref{EQ2}) and (\ref{EQ3}) are satisfied.

Let us consider such a $\varpi$.
For any $k,l\in \mathbb{N}$, we denote by $\varpi_{k,l}$ the restriction of $\varpi$ to $V^{\otimes k}\otimes V^{\otimes l}$.
By (\ref{EQ5}), $\varpi_{k,l}=0$ if $k+l\neq 1$. As (\ref{EQ2}) implies that $\varpi_{0,1}=0$, the only possibly nonzero $\varpi_{k,l}$
is $\varpi_{1,0}:V\longrightarrow V$, which we denote by $f$. Then (\ref{EQ1}) gives (\ref{EQ6}).

Let us consider any linear endomorphism $f$ of $V$ and consider $\varpi$ such that the only nonzero component of $\varpi$
is $\varpi_{1,0}=f$. Let us prove (\ref{EQ2})  for  $u\in V^{\otimes k}$, $v\in V^{\otimes l}$, $w\in V^{\otimes n}$, with $k+l+n=1-N$.
For all the possibilities for $(k,l,n)$, $0\in \{k,l,n\}$, and the result is then obvious.

Let us prove (\ref{EQ2})  for  $u\in V^{\otimes k}$, $v\in V^{\otimes l}$, $w\in V^{\otimes n}$, with $k+l+n=1-2N$. We obtain two possibilities:
\begin{itemize}
\item $(k,l,n)=(0,1,0)$ or $(0,0,1)$. We can assume that $u=\emptyset$. As $\emptyset\bullet x=0$ for any $x\in T(V)$, the result is obvious.
\item $(k,l,n)=(1,0,0)$. We can assume that $v=w=\emptyset$, and the result is then obvious.  \qedhere
\end{itemize} \end{proof}

\begin{remark} \begin{enumerate}
\item If $N\geq 1$, necessarily $f=0$, so $\bullet=0$.
\item With the notation of Proposition \ref{prop4}, $f_{T(V,f)}=f$.
\end{enumerate}\end{remark}

We obtain in this way the family of Com-PreLie bialgebras of \cite{Foissy28},
coming from a problem of composition of Fliess operators in Control Theory. Consequently, from \cite{Foissy28}:

\begin{cor}
Let $k,l\geq 0$. We denote by $\sh(k,l)$ the set of $(k,l)$-shuffles, that it to say permutations $\sigma \in \mathfrak{S}_{k+l}$ such that 
\begin{align*}
&\sigma(1)<\ldots<\sigma(k),&&\sigma(k+1)<\ldots<\sigma(k+l).
\end{align*}
If $\sigma\in \sh(k,l)$,  we put 
\begin{align*}
m_k(\sigma)&=\max\{i\in [k]\mid \sigma(1)=1,\ldots,\sigma(i)=i\},
\end{align*}
with the convention $m_k(\sigma)=0$ if $\sigma(1)\neq 1$. Then, in $T(V,f)$, if $v_1,\ldots,v_{k+l}\in V$,
\begin{align}\label{EQshuffle}
v_1\ldots v_k \bullet v_{k+1}\ldots v_{k+l}&=\sum_{\sigma\in \sh(k,l)}\sum_{i=1}^{m_k(\sigma)}
(\id^{\otimes (i-1)}\otimes f\otimes \id^{\otimes(k+l-i)})(v_{\sigma^{-1}(1)}\ldots v_{\sigma^{-1}(k+l)}).
\end{align} \end{cor}

\subsection{PreLie products of degree $-1$}

\begin{prop}\label{prop11}
Let $*$ and $\{-,-\}$ be two bilinear products on $V$ such that 
\begin{align}\label{EQ7}
&\forall x,y,z\in V,&(x*y)*z-x*(y*z)&=(x*z)*y-x*(z*y),\\
\nonumber&&\{x,y\}&=-\{y,x\},\\
\nonumber&&x*\{y,z\}&=\{x*y,z\},\\
\nonumber&&\{x,y\}*z&=\{x*z,y\}+\{x,y*z\}+\{\{x,y\},z\}.
\end{align}
We define a product $\bullet$ on $T(V)$ in the following way: for all $x_1,\ldots,x_m,y_1,\ldots,y_n\in V$,
\begin{align}\label{EQ8}
x_1\ldots x_m\bullet y_1\ldots y_n&=\sum_{i=1}^n x_1\ldots x_{i-1}(x_i*y_1)(x_{i+1}\ldots x_m\shuffle y_2\ldots y_n)\\
\nonumber&+\sum_{i=1}^{k-1}x_1\ldots x_{i-1}\{x_i,x_{i+1}\}(x_{i+2}\ldots x_m\shuffle y_1\ldots y_n).
\end{align}
Then $(T(V),\shuffle,\bullet,\Delta)$ is a Com-PreLie bialgebra, and we obtain in this way all the possible preLie products $\bullet$, homogeneous of degree
$-1$, such that $(T(V),\shuffle,\bullet,\Delta)$ is a Com-PreLie bialgebra.
\end{prop}

\begin{proof} Let us consider a linear map $\varpi:T(V)\otimes T(V)\longrightarrow V$, satisfying (\ref{EQ5}) for $N=-1$.
Denoting by $\varpi_{k,l}=\varpi_{\mid V^{\otimes k}\otimes V^{\otimes l}}$ for any $k,l$, the only possibly nonzero $\varpi_{k,l}$
are for $(k,l)=(2,0)$, $(1,1)$ and $(0,2)$. For all $x,y\in V$, we put 
\begin{align*}
x*y&=\varpi_{1,1}(x\otimes y),&\{x,y\}&=\varpi_{2,0}(xy\otimes \emptyset).
\end{align*}
(\ref{EQ2}) is equivalent to 
\begin{align*}
&\forall w\in V^{\otimes 2},& \varpi_{0,2}(\emptyset\otimes w)&=0,\\
&\forall x,y \in V,& \varpi_{2,0}((xy+yx)\otimes \emptyset)&=0.
\end{align*}
Hence, we now assume that $\varpi_{0,2}=0$, and we obtain that $(\ref{EQ2})$ is equivalent to (\ref{EQ7})-2.
The nullity of $\varpi_{0,2}$ and (\ref{EQ1}) give (\ref{EQ8}).\\

Let us now consider (\ref{EQ3}), with $u\in V^{\otimes k}$, $v\in V^{\otimes l}$, $w\in V^{\otimes n}$, $k+l+n=1-2N=3$.
By symmetry between $v$ and $w$, and by nullity of $\varpi_{0,l}$ for all $l$, we have to consider two cases:
\begin{itemize}
\item $k=l=n=1$. We put $u=x$, $v=y$, $w=z$, with $x,y,z\in V$. Then (\ref{EQ3}) is equivalent to 
\begin{align*}
(x*y)*z-x*(y*z)&=(x*z)*y-x*(z*y),
\end{align*}
that is to say to (\ref{EQ7})-1.
\item $k=1$, $l=2$, $z=0$. We put $u=x$, $v=yz$, $w=\emptyset$, with $x,y,z\in V$. Then (\ref{EQ3}) is equivalent to 
\begin{align*}
\{x*y,z\}-x*\{y,z\}&=0,
\end{align*}
that is to say to (\ref{EQ7})-3.
\item $k=2$, $l=1$, $z=0$. We put $u=xy$, $v=z$, $w=\emptyset$, with $x,y,z\in V$. Then (\ref{EQ3}) is equivalent to 
\begin{align*}
\{x*z,y\}+\{x,y*z\}+\{\{x,y\},z\}&=\{x,y\}*z,
\end{align*}
that is to say to (\ref{EQ7})-4.
\end{itemize}
We conclude with Proposition \ref{prop9}. \end{proof}

\begin{remark} \begin{enumerate}
\item In particular, $*$ is a preLie product on $V$, and  for all $x,y\in V$, $x\bullet y=x*y$.
\item If $x_1,\ldots,x_m\in V$,
\begin{align*}
x_1\ldots x_m\bullet \emptyset&=\sum_{i=1}^{m-1} x_1\ldots x_{i-1} \{x_i,x_{i+1}\}x_{i+2}\ldots x_m.
\end{align*}\end{enumerate}\end{remark}

\begin{example} \begin{enumerate}
\item If $*$ is a preLie product on $V$, we can take $\{-,-\}=0$, and (\ref{EQ7}) is satisfied.
Using the classification of preLie algebras of dimension 2 over $\mathbb{C}$ of \cite{Benes2009}, it is not difficult to show that if the dimension of $V$
is $1$ or $2$, then necessarily $\{-,-\}$ is zero.
\item If $*=0$, then (\ref{EQ7}) becomes 
\begin{align*}
&\forall x,y\in V,& \{x,y\}&=-\{y,x\},\\
&\forall x,y,z\in V,&\{\{x,y\},z\}&=0,
\end{align*}
that is say $(V,\{-,-\})$ is a nilpotent Lie algebra, which nilpotency order is 2.
\item Here is a family of examples where both $*$ and $\{-,-\}$ are nonzero. Let $V$ be 3-dimensional space, with basis $(x,y,z)$, and let 
$a$, $b$, $c$ be scalars. We consider the products given by the following arrays:
\begin{align*}
&\begin{array}{c|c|c|c}
\bullet&x&y&z\\
\hline x&x&y&z\\
\hline y&0&0&0\\
\hline z&0&0&0
\end{array}&
&\begin{array}{c|c|c|c}
\{-,-\}&x&y&z\\
\hline x& 0&ay+bz&cy+(1-a)z\\
\hline y&-ay-bz&0&0\\
\hline z&(a-1)z-cy&0&0
\end{array}&
\end{align*}
Then $(V,\bullet,\{-,-\})$ satisfies (\ref{EQ7}) if, and only if, $a^2-a+bc=0$, or equivalently,
\[(2a-1)^2+(b+c)^2-(b-c)^2=1.\]
This equation defines a hyperboloid of one sheet.
\end{enumerate}\end{example}

\section{Free Com-PreLie algebras and quotients}

\subsection{Description of free Com-PreLie algebras}

We described in \cite{Foissy28} free Com-PreLie algebras in terms of decorated rooted partitioned trees.
We now work with free unitary Com-PreLie algebras.

\begin{defi}\label{defiarbres}\begin{enumerate}
\item A \emph{partitioned forest} is a pair $(F,I)$ such that 
\begin{enumerate}
\item $F$ is a rooted forest (the edges of $F$ being oriented from the roots to the leaves).
The set of its vertices is denoted by $V(F)$.
\item $I$ is a partition of the vertices of $F$ with the following condition:
if $x,y$ are two vertices of $F$ which are in the same part of $I$, then either they are both roots, or they have the same direct ascendant.
\end{enumerate}
The parts of the partition are called \emph{blocks}.
\item We shall say that a partitioned forest $F$ is a \emph{partitioned tree} if all the roots are in the same block.
Note that in this case, one of the blocks of $F$ is the set of roots of $F$. By convention, the empty forest $\emptyset$ is considered
as a partitioned tree.
\item Let $\D$ be a set. A \emph{partitioned tree decorated by $\D$} is a triple $(T,I,d)$, where $(T,I)$ is a partitioned tree and $d$ is a map
from the set of vertices of $T$ into $\D$. For any vertex $x$ of $T$, $d(x)$ is called the \emph{decoration} of $x$.
\item The set of isoclasses of partitioned trees, included the empty tree, will be denoted by $\PT$. 
For any set $\D$, the set of isoclasses of partitioned trees decorated by $\D$ will be denoted by $\PT(\D)$;
the set of isoclasses of partitioned trees decorated by $\N\times \D$ will be denoted by $\UPT(\D)=\PT(\N\times \D)$.
\end{enumerate}\end{defi}

\begin{example} We represent partitioned trees by the underlying rooted forest, the blocks of cardinality $\geq 2$
being represented by horizontal edges of different colors. Here are the partitioned trees with $\leq 4$ vertices:
\[\emptyset;\tun;\tdeux,\hdeux;\ttroisun,\htroisun,\ttroisdeux,\htroisdeux=\htroistrois,\htroisquatre;
\tquatreun, \hquatreun=\hquatredeux,\hquatretrois,\tquatredeux=\tquatretrois,\hquatrequatre=\hquatrecinq,\tquatrequatre,\hquatresix,\tquatrecinq,\]
\[\hquatresept=\hquatrehuit,\hquatreneuf=\hquatredix,\hquatreonze=\hquatredouze,\hquatretreize,\hquatrequatorze=\hquatrequinze=\hquatreseize,
\hquatredixsept.\]
\end{example}

Let us fix a set $\D$. 

\begin{defi} \label{defi13} Let $T=(T,I,d)$ and $T'=(T',J,d') \in \UPT(\D)$. 
\begin{enumerate}
\item The partitioned tree $T\cdot T'$ is defined as follows.
\begin{enumerate}
\item As a rooted forest, $T \cdot T'$ is $TT'$.
\item We put $I=\{I_1,\ldots,I_k\}$ and $J=\{J_1,\ldots,J_l\}$
and we assume that the block of roots of $T$ is $I_1$ and the block of roots of $T'$ is $J_1$. The partition of the vertices of $T \cdot T'$ is
$\{I_1\sqcup J_1,I_2,\ldots,I_k,J_2,\ldots,J_l\}$.
\end{enumerate}
$(\UPT(\D),\cdot)$ is a commutative monoid, of unit $\emptyset$.
\item Let $s$ be a vertex of $T'$. 
\begin{enumerate}
\item We denote by $\bl(s)$ the set of blocks of $T$, children of $s$.
\item Let $b\in \bl(s)\sqcup \{*\}$. We denote by $T\bullet_{s,b} T'$ the partitioned tree obtained in this way:
\begin{itemize}
\item Graft $T'$ on $s$, that is to say add edges from $s$ to any root of $T'$.
\item If $b\in \bl(s)$, join the block $b$ and the block of roots of $T'$.
\end{itemize} 
\item Let $k\in \mathbb{Z}$. The decoration of $s$ is denoted by $(i,d)$. 
The element $T[k]_s\in \UPT(\D)\sqcup\{0\}$ is defined by the following:
\begin{itemize}
\item If $i+k\geq 0$, replace the decoration of $s$ by $(i+k,d)$.
\item If $i+k<0$, $T[k]_s=0$.
\end{itemize}
\end{enumerate} \end{enumerate}\end{defi}

\begin{example} Let $T=\tdeux$, $T'=\tun$. We denote by $r$ the root of $T$ and by $l$ the leaf of $T$. Then 
\begin{align*}
\tdeux\bullet_{r,*} \tun&=\ttroisun,&\tdeux\bullet_{r,\{l\}}\tun&=\htroisun,&\tdeux\bullet_{l,*}\tun&=\ttroisdeux.
\end{align*}\end{example}

\begin{lemma} \label{lemmederivation}
Let $A_+=(A_+,\cdot,\bullet)$ a Com-PreLie algebra, and $f:A_+\longrightarrow A_+$ be a linear map such that 
\begin{align*}
&\forall x,y\in A_+,& f(x\cdot y)&=f(x)\cdot y+x\cdot f(y),\\
&&f(x\bullet y)&=f(x)\bullet y+x\bullet f(y)
\end{align*}
We put $A=A_+\oplus \vect(\emptyset)$. Then $A$ is given a unitary Com-PreLie algebra structure, extending the one of $A_+$, by 
\begin{align*}
&&\emptyset\cdot \emptyset&=\emptyset,&\emptyset\bullet \emptyset&=0,\\
&\forall x\in A_+,&x\cdot \emptyset&=x,&\emptyset\cdot x&=x,\\
&&x\bullet \emptyset&=f(x),&\emptyset\bullet x&=0.
\end{align*}\end{lemma}

\begin{proof} Obviously, $(A,\cdot)$ is a commutative, unitary associative algebra. 
Let us prove the PreLie identity for $x,y,z\in A_+\sqcup \{\emptyset\}$. 
\begin{itemize}
\item If $x=\emptyset$, then $x\bullet (y\bullet z)=(x\bullet y)\bullet z=x\bullet (z\bullet y)=(x\bullet z)\bullet y=0$. We now assume that $x\in A_+$.
\item If $y=z=\emptyset$, then obviously the PreLie identity is satisfied.
\item If $y=\emptyset$ and $z\in A_+$, then 
\begin{align*}
x\bullet (y\bullet z)&=0,&(x\bullet y)\bullet z&=f(x)\bullet z,\\
x\bullet (z\bullet y)&=x\bullet f(z),&(x\bullet z)\bullet y&=f(x\bullet z).
\end{align*}
As $f$ is a derivation for $\bullet$, the PreLie identity is satisfied. By symmetry, it is also true if $y\in A_+$ and $z=\emptyset$.
\end{itemize}
Let us now prove the Leibniz identity for  $x,y,z\in A_+\sqcup \{\emptyset\}$. It is obviously satisfied if $x=\emptyset$ or $y=\emptyset$;
we assume that $x,y\in A_+$. If $z=\emptyset$, then 
\begin{align*}
(x\cdot y)\bullet z&=f(x\cdot y),& (x\bullet z)\cdot y&=f(x)\cdot y,& x\cdot (y\bullet z)&=x\cdot f(y).
\end{align*}
As $f$ is a derivation for $\cdot$, the Leibniz identity is satisfied. \end{proof}

\begin{prop}
Let $\UCP(\D)$ be the vector space generated by $\UPT(\D)$.
We extend $\cdot$ by bilinearity and the PreLie product $\bullet$ is defined by 
\begin{align*}
&\forall T,T'\in \UPT(\D),&T\bullet T'&=\begin{cases}
\displaystyle \sum_{s\in V(t)} T\bullet_{s,*} T'\mbox{ if }t\neq \emptyset,\\
\displaystyle \sum_{s\in V(t)} T[+1]_s\mbox{ if } t=\emptyset.
\end{cases}\end{align*}
Then $\UCP(\D)$ is the free unitary Com-PreLie algebra generated by the the elements $\tdun{$(0,d)$}\hspace{4mm}$, $d\in D$.
\end{prop}

\begin{proof} We denote by $\UCP_+(\D)$ the subspace of $\UCP(\D)$ generated by nonempty trees. By Proposition 18 in \cite{Foissy28},
this is the free Com-PreLie algebra generated by the elements $\tdun{$(k,d)$}\hspace{4mm}$, $k\in \N$, $d\in \D$. 
We define a map $f:\UCP_+(\D)\longrightarrow \UCP_+(\D)$ by 
\begin{align*}
&\forall T\in \UPT(\D)\setminus \{\emptyset\},& f(T)&=\sum_{s\in V(t)} T[+1]_s.
\end{align*}
This is a derivation for both $\cdot$ and $\bullet$; by lemma \ref{lemmederivation}, $\UCP(\D)$ is a unitary Com-PreLie algebra.\\

Observe that for all $d\in \D$, $k\in \N$,
\[\tdun{$(0,d)$}\hspace{4mm}\bullet \emptyset^{\times k}=\tdun{$(k,d)$}\hspace{4mm}.\]
Let $A$ be a unitary Com-PreLie algebra and, for all $d\in \D$, let $a_d\in A$. By Proposition 18 in \cite{Foissy28},
we define a unique Com-PreLie algebra morphism by
\[\theta:\left\{\begin{array}{rcl}
\UCP_+(\D)&\longrightarrow&A\\
\tdun{$(k,d)$}\hspace{4mm}&\longrightarrow&a_d\times 1_A^{\times k}.
\end{array}\right.\]
We extend it to $\UCP(\D)$ by sending $\emptyset$ to $1_A$, and we obtain in this way a unitary Com-PreLie algebra from $\UCP(\D)$
to $A$, sending $\tdun{$(0,d)$}\hspace{4mm}$ to $a_d$ for any $d\in \D$. This morphism is clearly unique. \end{proof}

\begin{example} Let $i,j,k \in \N$ and $d,e,f \in \D$.
\begin{align*}
\tdun{$(i,d)$}\hspace{4mm} \bullet\tdun{$(j,e)$}\hspace{4mm}&=\tddeux{$(i,d)$}{$(j,e)$}\hspace{4mm},\\
\tdun{$(i,d)$}\hspace{4mm} \bullet\hspace{4mm} \hddeux{\hspace{-4.5mm}$(j,e)$}{$(k,f)$}\hspace{4mm}
&=\hspace{4mm}\hdtroisun{$(i,d)$}{$(k,f)$}{\hspace{-4.5mm}$(j,e)$}\hspace{4mm}\\
\tdun{$(i,d)$}\hspace{4mm} \bullet\tddeux{$(j,e)$}{$(k,f)$}\hspace{4mm}&=\tdtroisdeux{$(i,d)$}{$(j,e)$}{$(k,f)$}\hspace{4mm},\\
\tddeux{$(i,d)$}{$(j,e)$}\hspace{4mm}\bullet \tdun{$(k,f)$}\hspace{4mm}&=\tdtroisdeux{$(i,d)$}{$(j,e)$}{$(k,f)$}\hspace{4mm}
+\hspace{4mm}\tdtroisun{$(i,d)$}{$(k,f)$}{\hspace{-5mm}$(j,e)$}\hspace{3mm},\\
\tdun{$(i,d)$}\hspace{4mm} \bullet \emptyset&=\tdun{$(i+1,d)$}\hspace{8.5mm},\\
\tddeux{$(i,d)$}{$(j,e)$}\hspace{4mm}\bullet \emptyset&=\tddeux{$(i+1,d)$}{$(j,e)$}\hspace{8.5mm}+\tddeux{$(i,d)$}{$(j+1,e)$}\hspace{8.5mm},\\
\hspace{4mm}\tdtroisun{$(i,d)$}{$(k,f)$}{\hspace{-5mm}$(j,e)$}\hspace{3mm}\bullet \emptyset&=
\hspace{4mm}\tdtroisun{$(i+1,d)$}{$(k,f)$}{\hspace{-5mm}$(j,e)$}\hspace{6mm}
+\hspace{8mm}\tdtroisun{$(i,d)$}{$(k,f)$}{\hspace{-9mm}$(j+1,e)$}\hspace{4mm}
+\hspace{4mm}\tdtroisun{$(i,d)$}{$(k+1,f)$}{\hspace{-5mm}$(j,e)$}\hspace{4mm}.
\end{align*}\end{example}

\subsection{Quotients of $\UCP(\D)$}

\begin{prop}\label{prop16}
We put $V_0=\vect(\tdun{$(0,d)$}\hspace{4mm},d\in \D)$, identified with $\vect(\tdun{$d$},d\in \D)$. 
Let $f:V_0\longrightarrow V_0$ be any linear map. 
We consider the Com-PreLie ideal $I_f$ of $\UCP(\D)$ generated by the elements $\tdun{$(1,d)$}\hspace{4mm}-f(\tdun{$(0,d)$}\hspace{4mm})$, $d\in \D$.
\begin{enumerate}
\item We denote by $\UPT'(\D)$ the set of trees $T\in \UPT(\D)$ such that for any vertex $s$ of $T$, the decoration of $s$ is of the form
$(0,d)$, with $d\in \D$. It is trivially identified with $\PT(\D)$. Then the family $(T+I_f)_{T\in \UPT'(\D)}$ is a basis of $\UCP(\D)/I_f$.
\item In $\UCP(\D)/I_f$, for any $d\in \D$, $\tdun{$(0,d)$}\hspace{4mm}\bullet \emptyset=f(\tdun{$(0,d)$}\hspace{4mm})$. 
\end{enumerate}\end{prop}

\begin{proof} \textit{First step}.  We fix $d\in \D$. Let us first prove that for all $k \geq 0$,
\[\tdun{$(k,d)$}\hspace{4mm}+I_f=f^k(\tdun{$(0,d)$}\hspace{4mm})+I_f.\]
It is obvious if $k=0,1$. 
Let us assume the result at rank $k-1$. We put $f(\tdun{$(0,d)$}\hspace{4mm})=\sum a_e \tdun{$(0,e)$}\hspace{4mm}$. Then 
\begin{align*}
\tdun{$(k,d)$}\hspace{4mm}+I_f&=\tdun{$(1,d)$}\hspace{4mm}\bullet \emptyset^{\times (k-1)}+I_f\\
&=\sum a_e \tdun{$(0,e)$}\hspace{4mm}\bullet \emptyset^{\times (k-1)}+I_f\\
&=\sum a_e f^{k-1}(\tdun{$(0,e)$}\hspace{4mm})+I_f\\
&=f^k(\tdun{$(0,d)$}\hspace{4mm})+I_f,
\end{align*}
so the result holds for all $k$.\\

\textit{Second step}. Let $T\in \UPT(\D)$; let us prove that there exists $x\in \vect(\UPT'(\D))$, such that $T+I_f=x+I_f$.
We proceed by induction on $|T|$. If $|T|=0$, then $t=\emptyset$ and we can take $x=T$. If $|T|=1$, then $T=\tdun{$(k,d)$}\hspace{4mm}$ and we can take,
by the first step, $x=f^k(\tdun{$(0,d)$}\hspace{4mm})$. Let us assume the result at all ranks $<|T|$. If $T$ has several roots,
we can write $T=T_1\cdot T_2$, with $|T_1|,|T_2|<|T|$. Hence, there exists $x_i\in \vect(\UPT'(\D))$, such that $T_i+I_f=x_i+I_f$ for all $i\in [2]$,
and we take $x=x_1\cdot x_2$. Otherwise, we can write 
\[T=\tdun{$(k,d)$}\hspace{4mm} \bullet T_1\times \ldots \times T_k,\]
where $T_1,\ldots,T_k\in \UPT(\D)$. By the induction hypothesis, there exists $x_i\in \vect(\UPT'(\D))$ such that $T_i+I_f=x_i+I_f$ for all $i\in [k]$. 
We then take $x=f^k(\tdun{$(0,d)$}\hspace{4mm})\bullet x_1\times \ldots \times x_k$. \\

\textit{Third step.} We give $\CP_+(\D)=\vect(\PT(\D)\setminus \{\emptyset\})$ a Com-PreLie structure by 
\begin{align*}
&\forall T,T'\in \PT(\D)\setminus \{\emptyset\},&T\bullet T'&=\sum_{s\in V(t)} T\bullet_{s,*} T'.
\end{align*}
We consider the map 
\[F:\left\{\begin{array}{rcl}
\CP_+(\D)&\longrightarrow&\CP_+(\D)\\
T&\longrightarrow&\displaystyle \sum_{s\in V(T)} f_s(T),
\end{array}\right.\]
where,  $f_s(T)$ is the linear span of decorated partitioned trees obtained by replacing the decoration $d_s$ of $s$ by $f(d_s)$, the trees
being considered as linear in any of their decorations. This is a derivation for both $\cdot$ and $\bullet$, so by lemma \ref{lemmederivation},
$\CP(\D)$ inherits a unitary Com-PreLie structure such that for any $d\in \D$,
\[\tdun{$d$}\bullet \emptyset=f(\tdun{$d$}).\]
By the universal property of $\UCP(\D)$, there exists a unique unitary Com-PreLie algebra morphism $\phi:\UCP(\D)\longrightarrow \CP(\D)$,
such that $\phi(\tdun{$(0,d)$}\hspace{4mm})=\tdun{$d$}$ for any $d\in \D$. Then $\phi(\tdun{$(1,d)$}\hspace{4mm})=f(\tdun{$d$}))=\phi(f(\tdun{$(0,d)$}\hspace{4mm})$ for any $d\in D$,
so $\phi$ induces a morphism $\overline{\phi}:\UCP(\D)/I_f\longrightarrow \CP(\D)$. It is not difficult to prove that for any $T\in \UPT'(\D)$,
$\phi(T)=T$. As the family $\PT(\D)$ is a basis of $\CP(\D)$, the family $(T+I_f)_{T\in \UPT'(\D)}$ is linearly independent in $\UCP(\D)/I_f$.
By the second step, it is a basis. \end{proof}

\begin{example} We choose $f=\id_{V_0}$. The product in $\UCP(\D)/{I_{\id_{V_0}}}$ is the one of Definition \ref{defi13}. 
If $T,T'\in \PT(\D)$ and $T'\neq \emptyset$, $T\bullet T'$ is the sum of all graftings of $T'$ over $T$. Moreover,
\[T\bullet \emptyset=|T|T.\]
Hence, we now consider $\CP(\D)$, augmented by an unit $\emptyset$, as a unitary Com-PreLie algebra.
\end{example}

\begin{prop}\label{prop17}
Let $J$ be the Com-PreLie ideal of $\CP(\D)$ generated by the elements 
\[\tdun{$d$} \bullet (F_1\times F_2)-
\tdun{$d$}\bullet (F_1\cdot F_2),\]
with $d\in \D$ and $F_1,F_2 \in \PT(\D)$. 
\begin{enumerate}
\item Let $T$ and $T'$ be two elements of $\PT(\D)$ which are equal as decorated rooted forests. Then $T+J=T'+J$. Consequently,
if $F$ is a decorated rooted forest, the element $T'+I$ does not depend of the choice of 
$T'\in \UPT(\D)$ such that $T'=F$ as a decorated rooted forest.  This element is identified with $F$.
\item The set of decorated rooted forests is a basis of $\UCP(\D)/J$. 
\end{enumerate}
$\CP(\D)/J$ is then, as an algebra, identified with the Connes-Kreimer algebra $H_{\CK}^{\D}$ of decorated rooted trees 
\cite{Connes98}, which is in this way a unitary Com-PreLie algebra.\end{prop}

\begin{proof} 1. \textit{First step.} Let us show that for any $x_1,\ldots,x_n \in \UCP(\D)$,
$\tdun{$d$}\bullet (x_1 \times \ldots \times x_n)+J=\tdun{$d$}\bullet (x_1\cdot \ldots \cdot x_n)+J$ by induction on $n$.
It is obvious if $n=1$, and it comes from the definition of $J$ if $n=2$. Let us assume the result at rank $n-1$. 
\begin{align*}
&\tdun{$d$}\bullet (x_1\times \ldots \times x_n)+J\\
&=(\tdun{$d$} \bullet (x_1\times \ldots \times x_{n-1}))\bullet x_n
-\sum_{i=1}^{n-1} \tdun{$d$}\bullet (x_1 \times \ldots \times (x_i \bullet x_n) \times \ldots \times x_{n-1})+J\\
&=(\tdun{$d$} \bullet (x_1\cdot \ldots \cdot x_{n-1}))\bullet x_n
-\sum_{i=1}^{n-1} \tdun{$d$}\bullet (x_1 \cdot \ldots \cdot (x_i \bullet x_n) \cdot \ldots \cdot x_{n-1})+J\\
&=(\tdun{$d$} \bullet (x_1\cdot \ldots \cdot x_{n-1}))\bullet x_n
-\tdun{$d$}\bullet ((x_1 \cdot \ldots \cdot x_{n-1})\bullet x_n)+J\\
&=\tdun{$d$}\bullet((x_1\cdot \ldots \cdot x_{n-1})\times x_n)+J\\
&=\tdun{$d$}\bullet (x_1\cdot \ldots x_{n-1}\cdot x_n)+J.
\end{align*}
So the result holds for all $n$. \\

\textit{Second step.} Let $F,G \in \PT(\D)$, such that the underlying rooted decorated forests are equal. Let us prove that $F+J=G+J$
by induction on $n=|F|=|G|$. If $n=0$, $F=G=1$ and it is obvious. If $n=1$, $F=G=\tdun{$d$}$ and it is obvious. Let us assume
the result at all ranks $<n$. 

\textit{First case}. If $F$ has $k \geq 2$ roots, we can write $F=T_1\cdot \ldots \cdot T_k$ and $G=T'_1\cdot \ldots \cdot T'_k$,
such that, for all $i\in [k]$, $T_i$ and $T'_i$ have the same underlying decorated rooted forest; By the induction hypothesis,
$T_i+J=T'_i+J$ for all $i$, so $F+J=G+J$.

\textit{Second case.} Let us assume that $F$ has only one root. We can write $F=\tdun{$d$}\bullet(F_1\times \ldots \times F_k)$
and $G=\tdun{$d$} \bullet (G_1\times \ldots \times G_l)$. Then $F_1\cdot \ldots \cdot F_k$ and $G_1\cdot \ldots \cdot G_l$
have the same underlying decorated forest; by the induction hypothesis, $F_1\cdot \ldots \cdot F_k+J=G_1\cdot \ldots \cdot G_l+J$,
so $\tdun{$d$}\bullet ( F_1\cdot \ldots \cdot F_k)+J=\tdun{$d$}\bullet(G_1\cdot \ldots \cdot G_l)+J$. By the first step,
\[F+J=\tdun{$d$}\bullet ( F_1\cdot \ldots \cdot F_k)+J=\tdun{$d$}\bullet(G_1\cdot \ldots \cdot G_l)+J=G+J.\]

2. The set $\RF(\D)$ of rooted forests linearly spans $\CP(\D)/J$ by the first point.
Let $J'$ be the subspace of $\CP(\D)$ generated by the differences of elements of $\PT(\D)$ with the same underlying decorated forest.
It is clearly a Com-PreLie ideal, and $\RF(\D)$ is a basis of $\CP(\D)/J'$. Moreover, for all $F_1,F_2 \in \PT(\D)$,
$\tdun{$d$}\bullet(F_1 \times F_2)+J'=\tdun{$s$}\bullet(F_1\cdot F_2)+J'$, as the underlying forests of 
$\tdun{$d$}\bullet(F_1 \times F_2)$ and $\tdun{$s$}\bullet(F_1\cdot F_2)$ are equal. Consequently, there exists a Com-PreLie morphism
from $\CP(\D)/J$ to $\CP(\D)/J'$, sending any element of $\RF(\D)$ over itself. As the elements of $RF(\D)$ are linearly independent in $\CP(\D)/J'$,
they also are in $\CP(\D)/J$.  \end{proof}

\subsection{PreLie structure of $\UCP(\D)$ and $\CP(\D)$}

Let us now consider $\UCP(\D)$ and $\CP(\D)$ as preLie algebras. Their augmentation ideals are respectively denoted by $\UCP_+(\D)$
and $\CP_+(\D)$. Note that, as a preLie algebra, $\UCP_+(\D)=\CP_+(\N\times \D)$.

Let $\D$ be any set, and let $T\in \PT(\D)$. Then $T$ can be written as 
\[T=\left(\tdun{$d_1$}\:\bullet (T_{1,1}\times\ldots \times T_{i,s_1})\right)\cdot \ldots \cdot 
\left(\tdun{$d_k$}\:\bullet (T_{k,1}\times\ldots \times T_{k,s_k})\right),\]
where $d_1,\ldots,d_k \in \D$ and the $T_{i,j}$'s are nonempty elements of $\PT(\D)$.
We shortly denote this as 
\[T=B_{d_1,\ldots,d_k}(T_{1,1}\ldots T_{1,s_1};\ldots;T_{k,1}\ldots T_{k,s_k}).\]
The set of partitioned subtrees $T_{i,j}$ of $T$ is denoted by $\st(T)$.

\begin{prop}
Let $\D$ be any set. One defines a coproduct $\delta$ on $\CP_+(\D)$ by 
\begin{align*}
&\forall T\in \PT(\D),&\delta(T)&=\sum_{T'\in \st(T)} T\setminus T'\otimes T.
\end{align*}
Then, as a preLie algebra, $\CP_+(\D)$ is freely generated by $\ker(\delta)$.
\end{prop}

\begin{proof}
In other words, for any $T\in \PT(\D)$, writing
\[T=B_{d_1,\ldots,d_k}(T_{1,1}\ldots T_{1,s_1};\ldots; T_{k,1}\ldots T_{k,s_k}).\]
we can rewrite
\[\delta(T)=\sum_{i=1}^s \sum_{j=1}^{s_i}
B_{d_1,\ldots,d_k}(T_{1,1}\ldots T_{1,s_1};\ldots; T_{i,1}\ldots \widehat{T_{i,j}}\ldots T_{i,s_i};\ldots;T_{k,1}\ldots T_{k,s_k})\otimes T_{i,j}.\]
This immediately implies that $\delta$ is permutative \cite{Livernet2006}:
\[(\delta\otimes \id)\circ \delta=(23). (\delta \otimes \id)\circ \delta.\]
Moreover, for any $x,y\in \PT_+(\D)$, using Sweedler's notation $\delta(x)=x^{(1)}\otimes x^{(2)}$, we obtain 
\[\delta(x\cdot y)=x^{(1)}\cdot y\otimes x^{(2)}+x\cdot y^{(1)}\otimes y^{(2)}.\]

For any partitioned tree $T\in \PT(\D)$, we denote by $r(T)$ the number of roots of $T$ and we put $d(T)=r(T)T$.
The map $d$ is linearly extended as an endomorphism of $\PT(\D)$. As the product $\cdot$ is homogeneous for the number of roots,
$d$ is a derivation of the algebra $(\CP(\D),\cdot)$. Let us prove that for any $x,y\in \CP_+(\D)$,
\[\delta(x\bullet y)=d(x)\otimes y+x^{(1)}\bullet y\otimes x^{(2)}+x^{(1)}\otimes x^{(2)}\bullet y.\]
We denote by $A$ the set of elements of $x\in \CP_+(\D)$, such that for any $y\in \CP_+(\D)$, the preceding equality holds.
If $x_1,x_2\in A$, then for any $y\in \CP_+(\D)$,
\begin{align*}
\delta((x_1\cdot x_2)\bullet y)&=\delta((x_1\bullet y)\cdot x_2)+\delta(x_1\cdot (x_2\bullet y))\\
&=(x_1\bullet y)^{(1)}\cdot x_2\otimes(x_1\bullet y)^{(2)}+(x_1\bullet y)\cdot x_2^{(1)}\otimes x_2^{(2)}\\
&+x_1^{(1)}\cdot (x_2\bullet y)\otimes x_1^{(2)}+x_1\cdot (x_2\bullet y)^{(1)}\otimes (x_2\bullet y)^{(2)}\\
&=d(x_1)\cdot x_2 \otimes y+(x_1^{(1)} \bullet y)\cdot x_2\otimes x_1^{(1)}+x_1^{(1)}\cdot x_2\otimes x_1^{(2)}\bullet y\\
&+(x_1\bullet y)\cdot x_2^{(1)}\otimes x_2^{(2)}+x_1^{(1)}\cdot (x_2\bullet y)\otimes x_1^{(2)}\\
&+x_1\cdot d(x_2)\otimes y+x_1\cdot (x_2^{(1)}\bullet y)\otimes x_2^{(2)}+x_1\cdot x_2^{(1)}\otimes x_2^{(2)}\bullet y\\
&=d(x_1\cdot x_2)\otimes y+(x_1^{(1)}\cdot x_2)\bullet y\otimes x_1^{(2)}+(x_1\cdot x_2^{(1)})\bullet y\otimes x_2^{(2)}\\
&+(x_1\cdot x_2)^{(1)}\otimes (x_1\cdot x_2)^{(2)}\bullet y\\
&=d(x_1\cdot x_2)\otimes y+(x_1\cdot x_2)^{(1)}\bullet y\otimes (x_1\cdot x_2)^{(2)}+(x_1\cdot x_2)^{(1)}\otimes (x_1\cdot x_2)^{(2)}\bullet y.
\end{align*}
So $x_1\cdot x_2\in A$. 

Let $d\in \D$. Note that $\delta(\tdun{$d$})=0$. Moreover, for any $y\in \CP_+(\D)$,
\begin{align*}
\delta(\tdun{$d$}\bullet y)&=\delta(B_d(y))=\tdun{$d$}\otimes y,
\end{align*}
so $\tdun{$d$}\in A$. Let $T_1,\ldots,T_k\in \PT(\D)$, nonempty. We consider $x=B_d(T_1\ldots T_k)$. For any $y\in \CP_+(\D)$,
\begin{align*}
\delta(x\bullet y)&=\delta(B_d(T_1\ldots T_ky))+\sum_{j=1}^k \delta(B_d(T_1\ldots (T_j\bullet y)\ldots T_k)\\
&=B_d(T_1\ldots T_k)\otimes y+\sum_{i=1}^k B_d(T_1\ldots \widehat{T_i}\ldots T_ky)\otimes T_i\\
&+\sum_{i=1}^k \sum_{j\neq i} B_d(T_1\ldots \widehat{T_i}\ldots (T_j\bullet y)\ldots T_k)\otimes T_i
+\sum_{i=1}^k B_d(T_1\ldots \widehat{T_i}\ldots T_k)\otimes T_i\bullet y\\
&=d(x)\otimes y+\sum_{i=1}^k B_d(T_1\ldots \widehat{T_i}\ldots T_k)\bullet y\otimes T_i
+\sum_{i=1}^k B_d(T_1\ldots \widehat{T_i}\ldots T_k)\otimes T_i\bullet y\\
&=d(x)\otimes y+x^{(1)}\bullet y\otimes x^{(2)}+x^{(1)}\otimes x^{(2)}\bullet y.
\end{align*}
Hence, $x\in A$. As $A$ is stable under $\cdot$ and contains any partitioned tree with one root, $A=\CP_+(\D)$.\\

For any nonempty partitioned tree $T\in \PT(\D)$, we put $\displaystyle \delta'(T)=\frac{1}{r(T)} \delta(T)$. Then 
\[(\delta' \otimes \id)\circ \delta'(T)=\frac{1}{r(T)^2} (\delta \otimes \id)\circ \delta(T),\]
so $\delta'$ is also permutative; moreover, for any $x,y\in \CP_+(\D)$,
\[\delta'(x\bullet y)=x\otimes y+x^{(1)}\bullet y\otimes x^{(2)}+x^{(1)}\otimes x^{(2)}\bullet y.\]
By Livernet's rigidity theorem \cite{Livernet2006}, the preLie algebra $\CP_+(\D)$ is freely generated by $\ker(\delta')$.
For any integer $n$, we denote by $\CP_n(\D)$ the subspace of $\CP(\D)$ generated by trees $T$ such that $r(T)=n$.
Then, for all $n$,  $\delta(\CP_n(\D))\subseteq \CP_n(\D)\otimes \CP_+(\D)$, and $\delta_{\mid \CP_n(\D)}=n\delta'_{\mid \CP_n(\D)}$.
This implies that $\ker(\delta)=\ker(\delta')$. \end{proof}

\begin{lemma}
In $\CP_+(\D)$ or $\UCP_+(\D)$, $\ker(\delta)\bullet \emptyset \subseteq \ker(\delta)$.
\end{lemma}

\begin{proof} Let us work in $\UCP_+(\D)$. Let us prove that for any $x\in \UCP_+(\D)$,
\[\delta(x\bullet \emptyset)=x^{(1)}\bullet \emptyset \otimes x^{(2)}+x^{(1)}\otimes x^{(2)}\bullet \emptyset.\]
We denote by $A$ the subspace of elements $x \in \UCP_+(\D)$ such that this holds. If $x_1,x_2\in A$, then 
\begin{align*}
\delta((x_1\cdot x_2)\bullet \emptyset)&=\delta((x_1\bullet \emptyset)\cdot x_2)+\delta(x_1\cdot (x_2\bullet \emptyset))\\
&=(x_1^{(1)}\bullet \emptyset) \cdot x_2\otimes x^{(1)}+x_1^{(1)}\cdot x_2 \otimes x_1^{(2)}\bullet \emptyset
+(x_1\bullet \emptyset)\cdot x_2^{(1)}\otimes x_2^{(2)}\\
&+x_1 \cdot (x_2^{(1)}\bullet \emptyset)\otimes x_2^{(2)}+x_1\cdot x_2^{(1)}\otimes x_2^{(2)}\bullet \emptyset
+x_1^{(1)}\cdot (x_2\bullet \emptyset)\otimes x_1^{(2)}\\
&=(x_1^{(1)}\cdot x_2)\bullet \emptyset\otimes x_1^{(2)}+x_1^{(1)}\cdot x_2 \otimes x_1^{(2)}\bullet \emptyset\\
&+(x_1\cdot x_2^{(1)})\bullet \emptyset \otimes x_2^{(1)}+x_1\cdot x_2^{(1)}\otimes x_2^{(2)}\bullet \emptyset\\
&=(x_1\cdot x_2)^{(1)}\bullet \emptyset\otimes (x_1\cdot x_2)^{(2)}+(x_1\cdot x_2)^{(1)}\otimes (x_1\cdot x_2)^{(2)}\bullet \emptyset,
\end{align*}
so $x_1\cdot x_2\in A$. If $d\in \D$ and $T_1,\ldots,T_k\in \UPT(\D)$, nonempty, if $x=B_d(T_1\ldots T_k)$,
\begin{align*}
\delta(x\bullet \emptyset)&=\delta (B_{d+1}(T_1\ldots T_k))+\sum_{i=1}^k \delta(B_d(T_1\ldots (T_i\bullet \emptyset)\ldots T_k)\\
&=\sum_{i=1}^k B_{d+1}(T_1\ldots \widehat{T_i}\ldots T_k)\otimes T_i
+\sum_{j=1}^k \sum_{i\neq j} B_d(T_1\ldots (T_j\bullet \emptyset)\ldots\widehat{T_i}\ldots T_k)\otimes T_i\\
&+\sum_{i=1}^k B_d(T_1\ldots \widehat{T_i}\ldots T_k)\otimes T_i\bullet \emptyset\\
&=\sum_{i=1}^k B_d(T_1\ldots \widehat{T_i}\ldots T_k)\bullet \emptyset \otimes T_i
+\sum_{i=1}^k B_d(T_1\ldots \widehat{T_i}\ldots T_k)\otimes T_i\bullet \emptyset\\
&=x^{(1)}\bullet \emptyset\otimes x^{(2)}+x^{(1)}\otimes x^{(2)}\bullet \emptyset,
\end{align*}
so $x\in A$. Hence, $A=\UCP_+(\D)$. Consequently, if $x\in \ker(\delta)$, then $x\bullet \emptyset \in \ker(\delta)$.
The proof is immediate for $\CP_+(\D)$, as for any tree $T\in\PT(\D)$, $T\bullet \emptyset=|T|T$. \end{proof}

\begin{notation}
We denote by $\phi$ the endomorphism of $\ker(\delta)$ defined by $\phi(x)=x\bullet \emptyset$.
\end{notation}

\begin{cor}
The preLie algebra $\UCP(\D)$, respectively $\CP(\D)$, is generated by $\ker(\delta)\oplus (\emptyset)$, with the relations 
\begin{align*}
&&\emptyset\bullet \emptyset&=0,\\
\forall& x\in \ker(\delta),&\emptyset \bullet x&=0,&x\bullet \emptyset&=\phi(x).
\end{align*}
\end{cor}

\begin{remark}
We give $\CP(\D)$ a graduation by putting the elements of $\D$ homogeneous of degree $1$, and we put $|D|=d$. for any $n\geq 1$, 
we denote by $t_n(d)$ the number of partitioned trees decorated by $\D$ with $n$ vertices
and by $f_n(d)$ the number of partitioned forests decorated by $\D$ with $n$ vertices. We consider the formal series
\begin{align*}
F(d,X)&=\sum_{n=0}^\infty f_n(d)X^n,&T(d,X)&=\sum_{n=0}^\infty t_n(d)X^n.
\end{align*}
As any partitioned forest is a monomial of partitioned trees, we obtain
\[F(d,X)=\prod_{n=1}^\infty \frac{1}{(1-X^n)^{t_n(d)}}.\]
As any partitioned tree can be seen as a monomial of pairs $(e,F)$, where $e\in \D$ et $F$ a partitioned forest, we obtain that
\[T(d,X)=\prod_{n=1}^\infty \frac{1}{(1-X^n)^{df_{n-1}(d)}}.\]
These two formulas allow to compute $t_n(d)$ by induction on $n$, see Table \ref{table1} (see also \cite{Foissy28}). For $d=1$, this gives Entry \href{https://oeis.org/A035052}{A035052} of the OEIS \cite{Sloane};
for $d=2$, Entry \href{https://oeis.org/A226269}{A226269}. Moreover,  the sequence of the coefficients of $\binom{d}{n}$ in $t_n(d)$ is Entry \href{https://oeis.org/A052888}{A052888}.

\begin{table}\label{table1}
\begin{align*}
t_1(d)&=d\\
&=\binom{d}{1},\\
t_2(d)&=\frac{(3d + 1)d}{2}\\
&=2\binom{d}{1} + 3\binom{d}{2},\\
t_3(d)&=\frac{(19d^2 + 9d + 2)d}{6}\\
&=5\binom{d}{1} + 22\binom{d}{2} + 19\binom{d}{3},\\
t_4(d)&=\frac{(63d^3 + 34d^2 + 13d + 2)d}{8}\\
&=14\binom{d}{1} + 139\binom{d}{2} + 309\binom{d}{3} + 189\binom{d}{4},\\
t_5(d)&=\frac{(644d^4 + 400d^3 + 175d^2 + 35d + 6)d}{30}\\
&=42\binom{d}{1} + 868\binom{d}{2} + 3735\binom{d}{3} + 5472\binom{d}{4} + 2576\binom{d}{5},\\
t_6(d)&=\frac{(44683d^5 + 31695d^4 + 14635d^3 + 4185d^2 + 1162d + 120)d}{720}\\
&=134\binom{d}{1} + 5491\binom{d}{2} + 40882\binom{d}{3} + 107866\binom{d}{4} + 116990\binom{d}{5} + 44683\binom{d}{6},\\
t_7(d)&=\frac{(941977d^6 + 754131d^5 + 375235d^4 + 125265d^3 + 35308d^2 + 5124d + 720)d}{5040}\\
&=42\binom{d}{1} + 4258\binom{d}{2} + 55452\binom{d}{3} + 243536\binom{d}{4} + 468055\binom{d}{5} + 408774\binom{d}{6} + 133036\binom{d}{7}.
\end{align*}
\caption{first values of $t_n(d)$}
\end{table}
We denote by $k_n(d)$ the dimension of $\ker(\tdelta)_n$ in $\CP(\D)$. 
As the preLie algebra $\CP(\D)$ is freely generated by $\ker(\tdelta)$, we obtain that
\[T(d)=\left(\sum_{n=1}^\infty k_n(d)X^n \right)\prod_{n=1}^\infty \frac{1}{(1-X^n)^{t_n(d)}},\] 
This allows to compute the first values of $k_n(d)$, see Table \ref{table2}.

\begin{table}
\label{table2}
\begin{align*}
k_1(d)&=d\\
&=\binom{d}{1},\\
k_2(d)&=\frac{(d + 1)d}{2}\\
&=\binom{d}{1} + \binom{d}{2},\\
k_3(d)&=\frac{(2d^2 + 1)d}{3}\\
&=\binom{d}{1} + 4\binom{d}{2} + 4\binom{d}{3},\\
k_4(d)&=\frac{(11d^3 + 2d^2 + d + 2)d}{8}\\
&=2\binom{d}{1} + 21\binom{d}{2} + 51\binom{d}{3} + 33\binom{d}{4},\\
k_5(d)&=\frac{(203d^4 + 60d^3 - 5d^2 - 30d + 12)d}{60}\\
&=4\binom{d}{1} + 114\binom{d}{2} + 543\binom{d}{3} + 836\binom{d}{4} + 406\binom{d}{5},\\
k_6(d)&=\frac{(220d^5+89d^4+16d^3+3d^2+4d+4)d}{24}\\
&=14\binom{d}{1} + 690\binom{d}{2} + 5531\binom{d}{3} + 15206\binom{d}{4} + 16945\binom{d}{5} + 6600\binom{d}{6},\\
k_7(d)&=\frac{(66518d^6 + 33831d^5 + 9170d^4 - 735d^3 - 1708d^2 - 1596d + 360)d}{2520}\\
&=42\binom{d}{1} + 4258\binom{d}{2} + 55452\binom{d}{3} + 243536\binom{d}{4} + 468055\binom{d}{5} + 408774\binom{d}{6} + 133036\binom{d}{7}.
\end{align*}
\caption{first values of $k_n(d)$}
\end{table}
\end{remark}

\section{Bialgebra structures on free Com-PreLie algebras}

\subsection{Tensor product of Com-PreLie algebras}

\begin{lemma}\label{lem23}
Let $A_1,A_2$ be two Com-PreLie algebras and let $\varepsilon:A_1\longrightarrow \K$ such that 
\begin{align*}
&\forall a,b\in A_1,&\varepsilon(a\bullet b)&=\varepsilon(b\bullet a).
\end{align*}
Then $A_1 \otimes A_2$ is a Com-PreLie algebra, with the products defined by 
\begin{align*}
(a_1\otimes a_2)(b_1\otimes b_2)&=a_1b_1\otimes a_2b_2,\\
(a_1\otimes a_2)\bullet_\varepsilon(b_1\otimes b_2)&=a_1\bullet b_1\otimes a_2b_2+\varepsilon(b_1) a_1 \otimes a_2\bullet b_2.
\end{align*}\end{lemma}

\begin{proof} $A_1\otimes A_2$ is obviously an associative and commutative algebra, with unit $1\otimes 1$. 
We take $\alpha=a_1\otimes a_2,\beta=b_1\otimes b_2,\gamma=c_1\otimes c_2 \in A_1\otimes A_2$. Let us prove the PreLie identity.
\begin{align*}
(\alpha\bullet_\varepsilon \beta) \bullet_\varepsilon \gamma-\alpha\bullet_\varepsilon (\beta \bullet_\varepsilon \gamma)&=
(a_1\bullet b_1)\bullet c_1 \otimes a_2b_2c_2+\varepsilon(c_1) a_1\bullet b_1 \otimes (a_2b_2)\bullet c_2\\
&+\varepsilon(b_1) a_1\bullet c_1\otimes (a_2\bullet b_2)c_2+\varepsilon(b_1)\varepsilon(c_1)a_1\otimes (a_2b\bullet _2)\bullet c_2\\
&-a_1\bullet (b_1\bullet c_1)\otimes a_2b_2c_2-\varepsilon(c_1)a_1\bullet b_1\otimes a_2(b_2\bullet c_2)\\
&-\varepsilon(c_1)\varepsilon(b_1) a_1\otimes a_2\bullet (b_2\bullet c_2)-\varepsilon(b_1\bullet c_1) a_1\otimes a_2\bullet (b_2c_2)\\
&=((a_1\bullet b_1)\bullet c_1-a_1\bullet (b_1\bullet c_1)) \otimes a_2b_2c_2\\
&+\varepsilon(b_1)\varepsilon(c_1) a_1\otimes ((a_2\bullet b_2)\bullet c_2-a_2\bullet (b_2\bullet c_2))\\
&+\varepsilon(c_1)a_1\bullet b_1\otimes (a_2\bullet c_2) b_2+\varepsilon(b_1)a_1\bullet c_1\otimes (a_2\bullet b_2) c_2\\
&-\varepsilon(b_1\bullet c_1) a_1\otimes a_2\bullet (b_2c_2).
\end{align*}
As $A_1$ and $A_2$ are PreLie, the first and second lines of the last equality are symmetric in $\beta$ and $\gamma$; the third line is obviously symmetric in $\beta$ and $\gamma$;
as $m$ is commutative and by the hypothesis on $\varepsilon$, the last line also is. So $\bullet_\varepsilon$ is PreLie.
\begin{align*}
(\alpha\beta)\bullet_\varepsilon \gamma&=(a_1b_1)\bullet c_1\otimes a_2b_2c_2+\varepsilon(c_1) a_1b_1\otimes (a_2b_2)\bullet c_2\\
&=((a_1\bullet c_1)b_1+a_1(b_1\bullet c_1))\otimes a_2b_2c_2
+\varepsilon(c_1) a_1b_1\otimes ((a_2\bullet c_2)b_2+a_2(b_2\bullet c_2))\\
&=(a_1\bullet c_1\otimes a_2c_2+\varepsilon(c_1)a_1\otimes a_2\bullet c_2)(b_1\otimes b_2)\\
&+(a_1\otimes a_2)(b_1\bullet c_1\otimes b_2c_2+\varepsilon(c_1)b_1\otimes b_2\bullet c_2)\\
&=(\alpha\bullet_\varepsilon \gamma)\beta+\alpha(\beta\bullet_\varepsilon \gamma).
\end{align*}
So $A_1\otimes A_2$ is Com-PreLie. \end{proof}

\begin{remark} Consequently, if $(A,m,\bullet,\Delta)$ is a Com-PreLie bialgebra, with counit $\varepsilon$, then $\Delta$ is a morphism of Com-PreLie algebras
from $(A,m,\bullet)$ to $(A\otimes A,m, \bullet_\varepsilon)$. Indeed, for all $a,b\in A$, $\varepsilon(a\bullet b)=\varepsilon(b\bullet a)=0$ and 
\begin{align*}
\Delta(a)\bullet_\varepsilon \Delta(b)&=a^{(1)}\bullet b^{(1)}\otimes a^{(2)}b^{(2)}
+\varepsilon(b^{(1)}) a^{(1)} \otimes a^{(2)}\bullet b^{(2)}\\
&=a^{(1)}\bullet b^{(1)}\otimes a^{(2)}b^{(2)}+ a^{(1)} \otimes a^{(2)}\bullet b\\
&=\Delta(a\bullet b).
\end{align*}\end{remark}

\begin{lemma}
\label{lem24}\begin{enumerate}
\item Let $A,B,C$ be three Com-PreLie algebras, $\varepsilon_A:A\longrightarrow \K$ and $\varepsilon_B:B\longrightarrow \K$
with the condition of lemma \ref{lem23}. Then $\varepsilon_A \otimes \varepsilon_B:A\otimes B\longrightarrow \K$ also satisfies the condition of
lemma \ref{lem23}. Moreover, the Com-PreLie algebras $(A\otimes B)\otimes C$ and $A\otimes (B\otimes C)$ are equal.
\item Let $A,B$ be two Com-PreLie algebras, and $\varepsilon:A\longrightarrow \K$ such that 
\begin{align*}
&\forall a,b\in A,&\varepsilon(ab)&=\varepsilon(a)\varepsilon(b),&\varepsilon(a\bullet b)&=0.
\end{align*}
Then $\varepsilon\otimes \id:A\otimes B\longrightarrow B$ is morphism of Com-PreLie algebras.
\item Let $A,A',B,B'$ be Com-PreLie algebras, $\varepsilon:A\longrightarrow \K$ and $\varepsilon':A'\longrightarrow \K$ satisfying the condition of 
lemma \ref{lem23}. Let $f:A\longrightarrow A'$, $g:B\longrightarrow B'$ be Com-PreLie algebra morphisms such that $\varepsilon'\circ f=\varepsilon$.
Then $f\otimes g:A\otimes B\longrightarrow A'\otimes B'$ is a Com-PreLie algebra morphism.
\end{enumerate}\end{lemma}

\begin{proof} 1. Indeed, if $a_1,a_2\in A$ and $b_1,b_2\in B$,
\begin{align*}
\varepsilon_A\otimes \varepsilon_B((a_1\otimes b_1)\bullet (a_2\otimes b_2))
&=\varepsilon_A(a_1\bullet a_2)\varepsilon_B(b_1b_2)+\varepsilon_A(a_1)\varepsilon_A(a_2)\varepsilon_B(b_1\bullet b_2)\\
&=\varepsilon_A(a_2\bullet a_1)\varepsilon_B(b_2b_1)+\varepsilon_A(a_2)\varepsilon_A(a_1)\varepsilon_B(b_2\bullet b_1)\\
&=\varepsilon_A\otimes \varepsilon_B((a_2\otimes b_2)\bullet (a_1\otimes b_1)).
\end{align*}
Let $a_1,a_2\in A$, $b_1,b_2\in B$, $c_1,c_2\in C$. In $(A\otimes B) \otimes C$,
\begin{align*}
&(a_1\otimes b_1\otimes c_1)\bullet(a_2\otimes b_2\otimes c_2)\\
&=((a_1\otimes b_1)\bullet (a_2\otimes b_2))\otimes c_1c_2
+\varepsilon_A\otimes \varepsilon_B(a_2\otimes b_2)a_1\otimes b_1\otimes c_1\bullet c_2\\
&=a_1\bullet a_2\otimes b_1b_2\otimes c_1c_2+\varepsilon_A(a_2) a_1\otimes b_1\bullet b_2\otimes c_1c_2
+\varepsilon_A(a_2)\varepsilon_B(b_2) a_1\otimes b_1\otimes c_1\bullet c_2.
\end{align*}
In $A\otimes (B \otimes C)$,
\begin{align*}
&(a_1\otimes b_1\otimes c_1)\bullet(a_2\otimes b_2\otimes c_2)\\
&=a_1\bullet a_2\otimes b_1b_2\otimes c_1c_2+\varepsilon_A(a_2)a_1\otimes ((b_1\otimes c_1)\bullet (b_2\otimes c_2))\\
&=a_1\bullet a_2\otimes b_1b_2\otimes c_1c_2+\varepsilon_A(a_2) a_1\otimes b_1\bullet b_2\otimes c_1c_2
+\varepsilon_A(a_2)\varepsilon_B(b_2) a_1\otimes b_1\otimes c_1\bullet c_2.
\end{align*}
So $(A\otimes B)\otimes C=A\otimes (B\otimes C)$.\\

2. Let $a_1,a_2\in A$, $b_1,b_2\in B$. 
\begin{align*}
&\varepsilon\otimes \id((a_1\otimes b_1)(a_2\otimes b_2))&&\varepsilon\otimes \id((a_1\otimes b_1)\bullet(a_2\otimes b_2))\\
&=\varepsilon(a_1a_2)b_1b_2&&=\varepsilon(a_1\bullet a_2)b_1b_2+\varepsilon(a_1)\varepsilon(a_2)b_1\bullet b_2\\
&=\varepsilon(a_1)\varepsilon(a_2)b_1b_2&&=\varepsilon(a_1)\varepsilon(a_2)b_1\bullet b_2\\
&=\varepsilon\otimes \id((a_1\otimes b_1)\varepsilon\otimes \id(a_2\otimes b_2),&
&=\varepsilon\otimes \id((a_1\otimes b_1)\bullet \varepsilon\otimes \id(a_2\otimes b_2).
\end{align*}
So $\varepsilon \otimes \id$ is a morphism. \\

3. $f\otimes g$ is obviously an algebra morphism. If $a_1,a_2\in A$, $b_1,b_2\in B$,
\begin{align*}
(f\otimes g)((a_1\otimes b_1)\bullet (a_2\otimes b_2))
&=(f\otimes g)(a_1\bullet a_2\otimes b_1b_2+\varepsilon(a_2)a_1\otimes b_1\bullet b_2)\\
&=f(a_1)\bullet f(a_2)\otimes g(b_1)g(b_2)+\varepsilon(f(a_2))f(a_1)\otimes g(b_1)\bullet g(b_2)\\
&=(f(a_1)\otimes g(b_1))\bullet (f(a_2)\otimes g(b_2)).
\end{align*}
So $f\otimes g$ is a Com-PreLie algebra morphism. \end{proof}

\begin{lemma}
Let $A$ be a unital associative commutative bialgebra, and $V$ a subspace of $A$ which generates $A$. 
Let $\bullet$ be a product on $A$ such that 
\begin{align*}
&\forall a,b,c \in A,&(ab)\bullet c&=(a\bullet c)b+a(b\bullet c).
\end{align*}
Then $A$ is a Com-PreLie bialgebra if, and only if, for all $x\in V$, and for all $b,c \in A$,
\begin{align*}
(x\bullet b)\bullet c-x\bullet (b\bullet c)&=(x\bullet c)\bullet b-x\bullet (c\bullet b),\\
\Delta(x\bullet b)&=x^{(1)}\otimes x^{(2)}\bullet b+x^{(1)}\bullet b^{(1)}\otimes x^{(2)}b^{(2)}.
\end{align*}\end{lemma}

\begin{proof} $\Longrightarrow$. Obvious, by definition of a Com-PreLie algebra. $\Longleftarrow$.
 We consider 
\[B=\{ a\in A\mid \forall b,c\in A, (a\bullet b)\bullet c-a\bullet (b\bullet c)=(a\bullet c)\bullet b-a\bullet (c\bullet b)\}.\]
We denote by $1_A$ the unit of $A$. 
Copying the proof of lemma \ref{lem3}-1, we obtain that $1_A.b=0$ for all $b\in A$. This easily implies that $1_A\in B$. By hypothesis, $V \subseteq B$.
Let $a_1,a_2 \in B$. For all $b,c \in A$,
\begin{align*}
&((a_1a_2)\bullet b)\bullet c-(a_1a_2)\bullet (b \bullet c)\\
&=((a_1\bullet b) \bullet c) a_2+(a_1\bullet b)(a_2\bullet c)
+(a_1\bullet c)(a_2\bullet b)+a_1((a_2\bullet b)\bullet c)\\
&-(a_1\bullet (b\bullet c))a_2-a_1(a_2\bullet (b\bullet c))\\
&=((a_1\bullet b) \bullet c-a_1\bullet (b\bullet c)) a_2+a_1((a_2\bullet b)\bullet c-a_2\bullet (b\bullet c))\\
&+(a_1\bullet b)(a_2\bullet c)+(a_1\bullet c)(a_2\bullet b).
\end{align*}
As $a_1,a_2\in B$, this is symmetric in $b,c$, so $a_1a_2\in B$. Hence, $B$ is a unitary subalgebra of $A$ which contains $V$, so is equal to $A$: $A$ is a Com-PreLie algebra.
Let us now consider 
\[C=\{a\in A\mid \forall b\in A, \Delta(a\bullet b)=a^{(1)}\otimes a^{(2)}\bullet b+a^{(1)}\bullet b^{(1)}\otimes a^{(2)}b^{(2)}\}.\]
By hypothesis, $V \subseteq C$. Let $b\in B$.
\[1_A\otimes 1_A\bullet b+1_A\bullet b^{(1)}\otimes 1b^{(2)}=0=\Delta(1_A\bullet b),\]
so $1_A\in C$. Let $a_1,a_2\in C$. For all $b\in A$,
\begin{align*}
\Delta((a_1a_2)\bullet b)&=\Delta((a_1\bullet b)a_2+a_1(a_2\bullet b))\\
&=a_1^{(1)}a_2^{(1)}\otimes  (a_1^{(2)}\bullet b)a_2^{(2)}+(a_1^{(1)}\bullet b^{(1)})a_2^{(1)}
\otimes a_1^{(2)} b^{(2)}a_2^{(2)}\\
&a_1^{(1)}a_2^{(1)}\otimes  a_1^{(2)}(a_2^{(2)}\bullet b)+a_1^{(1)}(a_2^{(1)}\bullet b^{(1)})
\otimes a_1^{(2)}a_2^{(2)} b^{(2)}\\
&=a_1^{(1)}a_2^{(1)}\otimes  (a_1^{(2)}a_2^{(2)})\bullet b+(a_1^{(1)}a_2^{(1)})\bullet b^{(1)}
\otimes a_1^{(2)}a_2^{(2)} b^{(2)}\\
&=(a_1a_2)^{(1)}\otimes (a_1a_2)^{(2)}\bullet b+(a_1a_2)^{(1)}\bullet b^{(1)}\otimes (a_1a_2)^{(2)}b^{(2)}.
\end{align*}
Hence, $a_1a_2\in C$, and $C$ is a unitary subalgebra of $A$. As it contains $V$, $C=A$ and $A$ is a Com-PreLie bialgebra. \end{proof}

\subsection{Coproduct on $\UCP(\D)$}

\begin{defi}
\begin{enumerate}
\item Let $T$ be a partitioned tree and $I\subseteq V(T)$. We shall say that $I$ is an  ideal of $T$ if for any vertex $v\in I$
and any vertex $w \in V(T)$ such that there exists an edge from $v$ to $w$, then $w\in I$. The set of ideals of $T$ is denoted $\Id(T)$.
\item Let $T$ be partitioned forest decorated by $\N\times I$, and $I\in \Id(T)$. 
\begin{itemize}
\item By restriction, $I$ is a partitioned decorated forest. The product $\cdot$ of the trees of $I$  is denoted by $P^I(F)$.
\item By restriction, $T\setminus I$ is a partitioned decorated tree. For any vertex $v \in T\setminus I$, if we denote by $(i,d)$ the decoration of $v$ in $T$,
we replace it by $(i+\iota_I(v),d)$, where $\iota_I(v)$ is the number of blocks $C$ of $T$, included in $I$,
such that there exists an edge from $v$ to any vertex of $C$.
The partitioned decorated tree obtained in this way is denoted by $R^I(F)$.
\end{itemize} \end{enumerate}\end{defi}

\begin{theo}
We define a coproduct on $\UCP(\D)$ by
\begin{align*}
&\forall T\in \PT(\N\times \D),&\Delta(T)&=\sum_{I\in \Id(T)} R^I(T)\otimes P^I(T).
\end{align*}
Then $\UCP(\D)$ is a Com-PreLie bialgebra. Moreover, $\CP(\D)$ and $\h_{\CK}^\D$ are Com-PreLie bialgebra quotients of $\UCP(\D)$,
and $\h_{\CK}^\D$ is the Connes-Kreimer Hopf algebra of decorated rooted trees \cite{Connes98,Foissy3}. 
\end{theo}

\begin{proof} We consider 
\[\varepsilon:\left\{\begin{array}{rcl}
\UCP(\D)&\longrightarrow&\K\\
F&\longrightarrow&\delta_{F,1}.
\end{array}\right.\]
By lemma \ref{lem24}-1, $\UCP(\D) \otimes_\varepsilon \UCP(\D)$ is a Com-PreLie algebra. It is unitary, the unit being $\emptyset\otimes \emptyset$.
Hence, there exists a unique Com-PreLie algebra morphism $\Delta':\UCP(\D)\longrightarrow \UCP(\D) \otimes_\varepsilon \UCP(\D)$,
sending $\tdun{$(0,d)$}\hspace{4mm}$ over $\tdun{$(0,d)$}\hspace{4mm} \otimes \emptyset+\emptyset\otimes \tdun{$(0,d)$}\hspace{4mm}$ for all $d\in \D$. 
By lemma \ref{lem24}-2, $(\UCP(\D) \otimes_\varepsilon \UCP(\D)) \otimes_{\varepsilon \otimes \varepsilon} UPC(\D)$
and $\UCP(\D) \otimes_\varepsilon (\UCP(\D) \otimes_\varepsilon \UCP(\D))$ are equal, and as both $(\id \otimes \Delta')\circ \Delta'$
and $(\Delta' \otimes \id)\circ \Delta'$ are Com-PreLie algebra morphisms sending $\tdun{$(0,d)$}\hspace{4mm}$ over
$\tdun{$(0,d)$}\hspace{4mm} \otimes \emptyset\otimes \emptyset+\emptyset\otimes \tdun{$(0,d)$}\hspace{4mm}\otimes \emptyset+\emptyset\otimes \emptyset\otimes \tdun{$(0,d)$}\hspace{4mm}$ for all $d\in \D$, they are equal:
$\Delta'$ is coassociative. Moreover, $(\id \otimes \varepsilon)\circ \Delta'$ and $(\varepsilon \otimes \id)\circ \Delta'$ are Com-PreLie endomorphisms
of $\UCP(\D)$ sending $\tdun{$(0,d)$}\hspace{4mm} $ over itself for all $d\in \D$, so they are both equal to $\id$: $\varepsilon$ is the counit of $\Delta'$.
Hence, with this coproduct $\Delta'$, $\UCP(\D)$ is a Com-PreLie bialgebra.\\

Let us now prove that $\Delta(\T)=\Delta'(\T)$ for all $T\in \PT(\N\times \D)$. We proceed by induction on the number of vertices $n$ of $T$.
If $n=0$ or $n=1$, it is obvious. Let us assume the result at all ranks $<n$. If $T$ has strictly more than one root, we can write
$T=T'\cdot T''$, where $T'$ and $T''$ has strictly less that $n$ vertices. It is easy to see that the ideals of $T$ are the parts of $T'\sqcup T''$ of the form
$I' \sqcup I''$, such that $I'\in \Id(\T')$ and $I''\in \Id(\T'')$. Moreover, for such an ideal of $T$, 
\begin{align*}
R^{I'\sqcup I''}(T'\cdot T'')&=R^{I'}(T')\cdot R^{I''}(T''),&P^{I'\sqcup I''}(T'\cdot T'')&=P^{I'}(T')\cdot P^{I''}(T'').
\end{align*}
Hence,
\begin{align*}
\Delta(T)&=\sum_{I'\in \Id(\T'),\: I''\in \Id(\T'')} R^{I'}(T')\cdot  R^{I''}(T'')\otimes R^{I'}(T') R^{I''}(T'')\\
&=\Delta(T) \cdot \Delta(T'')\\
&=\Delta'(T')\cdot \Delta'(T'')\\
&=\Delta'(T\cdot T'')\\
&=\Delta(T).
\end{align*}
If $T$ has only one root, we can write $T=\tdun{$(i,d)$}\hspace{4mm} \bullet(T_1\times\ldots \times T_k)$, 
where $T_1,\ldots,T_k\in \PT(\N\times \D)$. The induction hypothesis holds for $T_1,\ldots,T_N$. The ideals of $T$ are 
\begin{itemize}
\item $T$ itself: for this ideal $I$, $P^I(T)=T$ and $R^I(T)=\emptyset$.
\item Ideals $I_1\sqcup \ldots \sqcup I_k$, where $I_j$ is an ideal of $T_j$ for all $j$. For such an ideal $I$, 
$P^I(T)=P^{I_1}(T_1)\cdot \ldots \cdot P^{I_k}(T_k)$. Let $J=\{i_1,\ldots ,i_p\}$ be the set of indices $i$ such that $I_i=T_i$, 
that is to say the number of blocks $C$ of $I$ such that is an edge from the root of $T$
to any vertex of $C$. Then 
\begin{align*}
R^I(T)&=\tdun{$(i+p,d)$}\hspace{8mm}\bullet \prod_{j\notin J}^\times R^{I_j}(T_j)\\
&=f_{\UCP(\D)}^l(\tdun{$(i,d)$}\hspace{4mm})\bullet\prod_{j\notin J}^\times R^{I_j}(T_j)\\
&=\tdun{$(i,d)$}\hspace{4mm}\bullet \emptyset^{\times p}\times t\prod_{j\notin J}^\times R^{I_j}(T_j)\\
&=\tdun{$(i,d)$}\hspace{4mm} \bullet R^{I_1}(T_1)\times \ldots \times R^{I_k}(T_k).
\end{align*}
We used lemma \ref{lemmacoproduit} for the third equality.
\end{itemize}
By Proposition \ref{prop6}, with $a=\tdun{$(i,d)$}\hspace{4mm}$ and $b_1\times \ldots \times b_n=T_1\times \ldots \times T_k$,
\begin{align*}
\Delta'(T)&=\sum_{I\subseteq [k]} \tdun{$(i,d)$}\hspace{4mm}\bullet \left(\prod_{i\in I}^\times T_i^{(1)}\right)
\otimes \left(\prod_{i\in I} T_i^{(2)}\right) \emptyset\bullet\left(\prod_{i\notin I}^\times T_i\right)\\
&+\sum_{I\subseteq [k]}\emptyset \bullet \left(\prod_{i\in I}^\times T_i^{(1)}\right)
\otimes \left(\prod_{i\in I} T_i^{(2)}\right)\tdun{$(i,d)$}\hspace{4mm}\bullet\left(\prod_{i\notin I}^\times T_i\right)\\
&=\tdun{$(i,d)$}\hspace{4mm}\bullet T_1^{(1)}\times \ldots \times T^{(1)}_k\otimes T_1^{(2)}\cdot \ldots \cdot T_k^{(2)}+0\\
&+\emptyset\otimes \tdun{$(i,d)$}\hspace{4mm} \bullet T_1\times \ldots \times T_k\\
&=\sum_{I_j \in \Id(T_j)}\tdun{$(i,d)$}\hspace{4mm}\bullet R^{I_1}(T_1)\times \ldots \times R^{I_k}(T_k)\otimes 
P^{I_1}(T_1)\cdot \ldots \cdot P^{I_k}(T_k)+\emptyset \otimes T\\
&=\sum_{I\in \Id(T),\: I\neq T} R^I(T)\otimes P^I(T)+\emptyset \otimes T\\
&=\sum_{I\in \Id(T)} R^I(T)\otimes P^I(T)\\
&=\Delta(T).
\end{align*}
Hence, $\Delta'=\Delta$.\\

For all $d\in \D$, $\tdun{$(0,d)$}\hspace{4mm}-\tdun{$(1,d)$}\hspace{4mm}$ is primitive, 
so $\Delta(\tdun{$(0,d)$}\hspace{4mm}-\tdun{$(1,d)$}\hspace{4mm})\in I\otimes \UCP(\D)+\UCP(\D) \otimes I$.
Consequently, $I$ is a coideal, and the quotient $\UCP(\D)/I=\CP(\D)$ is a Com-PreLie bialgebra. 

Let $x,y \in \CP(\D)$. By Proposition \ref{prop6}, as $\tdun{$d$} $ is primitive,
\[\Delta(\tdun{$d$}\bullet(x \times y))=\tdun{$d$}\bullet (x^{(1)}\times y^{(1)})\otimes x^{(2)}\cdot y^{(2)}+
1\otimes \tdun{$d$}\bullet(x \times y),\]
whereas, by the 1-cocycle property,
\[\Delta(\tdun{$d$}\bullet(x \cdot y))=\tdun{$d$}\bullet (x^{(1)}\cdot y^{(1)})\otimes x^{(2)}\cdot y^{(2)}+
\otimes \tdun{$d$}\bullet(x \cdot y).\]
Hence,
\begin{align*}
\Delta(\tdun{$d$}\bullet(x \times y)-\tdun{$d$}\bullet(x \cdot y))
&=\underbrace{(\tdun{$d$}\bullet (x^{(1)}\times y^{(1)})-\tdun{$d$}\bullet (x^{(1)}\cdot y^{(1)}))}_{\in J}\otimes x^{(2)}\cdot y^{(2)}\\
&+1\otimes \underbrace{ (\tdun{$d$}\bullet(x \times y)-\tdun{$d$}\bullet(x \cdot y))}_{\in J}\\
&\in J\otimes \CP(\D)+\CP(\D)\otimes J,
\end{align*}
so $J$ is a coideal and $\CP(\D)/J=\h_{\CK}^\D$ is a Com-PreLie bialgebra.\\

Let us consider 
\[B_d:\left\{\begin{array}{rcl}
\h_{\CK}^\D&\longrightarrow&\h_{\CK}^\D\\
T_1\ldots T_k&\longrightarrow&\tdun{$d$}\bullet T_1\times \ldots \times T_k,
\end{array}\right.\]
where $T_1,\ldots, T_k$ are rooted trees decorated by $\D$. In other terms, $B_d(T_1\ldots T_k)$ is the tree obtained by grafting the forest 
$T_1\ldots T_k$ on a common root decorated by $d$. By Proposition \ref{prop6} and lemma \ref{lemmacoproduit}, 
for all forest $F=T_1\ldots T_k\in \h_{\CK}^\D$,
\begin{align*}
\Delta\circ B_d(F)&=\tdun{$d$}\bullet T_1^{(1)}\times \ldots \times T^{(1)}_k\otimes T_1^{(2)} \ldots T_k^{(2)}+0
+\emptyset\otimes \tdun{$d$} \bullet T_1\times \ldots \times T_k\\
&=B_d(F^{(1)})\otimes F^{(2)}+\emptyset \otimes B_d(F).
\end{align*}
We recognize the $1$-cocycle property which characterizes the Connes-Kreimer coproduct of rooted trees, so $\h_{\CK}^\D$ is indeed the 
Connes-Kreimer Hopf algebra. \end{proof}

\begin{example} Let $i,j,k \in \N$ and $d,e,f\in \D$. In $\UCP(\D)$,
\begin{align*}
\Delta\tdun{$(i,d)$}\hspace{4mm}&=\tdun{$(i,d)$}\hspace{4mm}\otimes \emptyset+\emptyset\otimes \tdun{$(i,d)$}\hspace{4mm},\\
\Delta\tddeux{$(i,d)$}{$(j,e)$}\hspace{4mm}&=\tddeux{$(i,d)$}{$(j,e)$}\hspace{4mm}\otimes \emptyset
+\emptyset\otimes \tddeux{$(i,d)$}{$(j,e)$}\hspace{4mm}+ \tdun{$(i+1,d)$}\hspace{9mm}\otimes  \tdun{$(j,e)$}\hspace{4mm},\\
\Delta\hspace{4mm}\tdtroisun{$(i,d)$}{$(k,f)$}{\hspace{-5mm}$(j,e)$}\hspace{3mm}&=
\hspace{4mm}\tdtroisun{$(i,d)$}{$(k,f)$}{\hspace{-5mm}$(j,e)$}\hspace{3mm}\otimes \emptyset
+\emptyset\otimes \hspace{4mm}\tdtroisun{$(i,d)$}{$(k,f)$}{\hspace{-5mm}$(j,e)$}\hspace{3mm}\\
&+\tddeux{$(i+1,d)$}{$(j,e)$}\hspace{9mm}\otimes \tdun{$(k,f)$}\hspace{5mm}+
\tddeux{$(i+1,d)$}{$(k,f)$}\hspace{9mm}\otimes \tdun{$(j,e)$}\hspace{5mm}+
\tdun{$(i+2,d)$}\hspace{9mm}\otimes \hspace{5mm}\hddeux{\hspace{-5mm}$(j,e)$}{$(k,f)$}\hspace{5mm},\\
\Delta\hspace{4mm}\hdtroisun{$(i,d)$}{$(k,f)$}{\hspace{-5mm}$(j,e)$}\hspace{3mm}&=
\hspace{4mm}\hdtroisun{$(i,d)$}{$(k,f)$}{\hspace{-5mm}$(j,e)$}\hspace{3mm}\otimes \emptyset
+\emptyset\otimes \hspace{4mm}\hdtroisun{$(i,d)$}{$(k,f)$}{\hspace{-5mm}$(j,e)$}\hspace{3mm}\\
&+\tddeux{$(i,d)$}{$(j,e)$}\hspace{5mm}\otimes \tdun{$(k,f)$}\hspace{5mm}+
\tddeux{$(i,d)$}{$(k,f)$}\hspace{5mm}\otimes \tdun{$(j,e)$}\hspace{5mm}+
\tdun{$(i+1,d)$}\hspace{9mm}\otimes \hspace{5mm}\hddeux{\hspace{-5mm}$(j,e)$}{$(k,f)$}\hspace{5mm},\\
\Delta\tdtroisdeux{$(i,d)$}{$(j,e)$}{$(k,f)$}\hspace{4mm}&=\tdtroisdeux{$(i,d)$}{$(j,e)$}{$(k,f)$}\hspace{4mm}
\otimes \emptyset +\emptyset\otimes\tdtroisdeux{$(i,d)$}{$(j,e)$}{$(k,f)$}\hspace{4mm}
+\tddeux{$(i,d)$}{$(j+1,e)$}\hspace{9mm}\otimes \tdun{$(k,f)$}\hspace{4mm}
+\tdun{$(i+1,d)$}\hspace{9mm}\otimes \tddeux{$(j,e)$}{$(k,f)$}\hspace{4mm}.
\end{align*}
In $\CP(\D)$,
\begin{align*}
\Delta\tdun{$d$}&=\tdun{$d$}\otimes \emptyset+\emptyset\otimes \tdun{$d$},\\
\Delta\tddeux{$d$}{$e$}&=\tddeux{$d$}{$e$}\otimes \emptyset+\emptyset\otimes \tddeux{$d$}{$e$}+ \tdun{$d$}\otimes  \tdun{$e$},\\
\Delta\tdtroisun{$d$}{$f$}{$e$}&=\tdtroisun{$d$}{$f$}{$e$}\otimes \emptyset+\emptyset\otimes \tdtroisun{$d$}{$f$}{$e$}
+\tddeux{$d$}{$e$}\otimes \tdun{$f$}+\tddeux{$d$}{$f$}\otimes \tdun{$e$}+\tdun{$d$}\otimes\hddeux{$e$}{$f$},\\
\Delta\hdtroisun{$d$}{$f$}{$e$}&=\hdtroisun{$d$}{$f$}{$e$}\otimes \emptyset+\emptyset\otimes \hdtroisun{$d$}{$f$}{$e$}
+\tddeux{$d$}{$e$}\otimes \tdun{$f$}+\tddeux{$d$}{$f$}\otimes \tdun{$e$}+\tdun{$d$}\otimes\hddeux{$e$}{$f$},\\
\Delta\tdtroisdeux{$d$}{$e$}{$f$}&=\tdtroisdeux{$d$}{$e$}{$f$}\otimes \emptyset +\emptyset\otimes\tdtroisdeux{$d$}{$e$}{$f$}
+\tddeux{$d$}{$e$}\otimes \tdun{$f$}+\tdun{$d$}\otimes \tddeux{$e$}{$f$}.
\end{align*}
In $\h_{\CK}^\D$,
\begin{align*}
\Delta\tdun{$d$}&=\tdun{$d$}\otimes \emptyset+\emptyset\otimes \tdun{$d$},\\
\Delta\tddeux{$d$}{$e$}&=\tddeux{$d$}{$e$}\otimes \emptyset+\emptyset\otimes \tddeux{$d$}{$e$}+ \tdun{$d$}\otimes  \tdun{$e$},\\
\Delta\tdtroisun{$d$}{$f$}{$e$}&=\tdtroisun{$d$}{$f$}{$e$}\otimes \emptyset+\emptyset\otimes \tdtroisun{$d$}{$f$}{$e$}
+\tddeux{$d$}{$e$}\otimes \tdun{$f$}+\tddeux{$d$}{$f$}\otimes \tdun{$e$}+\tdun{$d$}\otimes \tdun{$e$}\tdun{$f$},\\
\Delta\tdtroisdeux{$d$}{$e$}{$f$}&=\tdtroisdeux{$d$}{$e$}{$f$}\otimes \emptyset +\emptyset\otimes\tdtroisdeux{$d$}{$e$}{$f$}
+\tddeux{$d$}{$e$}\otimes \tdun{$f$}+\tdun{$d$}\otimes \tddeux{$e$}{$f$}.
\end{align*}\end{example}

\subsection{An application: Connes-Moscovici subalgebras}

Let us fix a set $\D$ of decorations. For any $d\in \D$, we define an operator $N_d:\h_{\CK}^\D\longrightarrow \h_{\CK}^\D$ by 
\begin{align*}
&\forall x\in \h_{\CK}^\D,&N_d(x)&=x\bullet \tdun{$d$}.
\end{align*}
In other words, if $F$ is a rooted forest, $N_d(F)$ is the sum of all forests obtained by grafting a leaf decorated by $d$ on a vertex of $F$:
when $\D$ is reduced to a singleton, this is the growth operator $N$ of \cite{Connes98}.

For all $k\geq 1$, $i_1,\ldots,i_k\in \D$, we put 
\begin{align*}
X_{i_1,\ldots,i_k}&=N_{i_k}\circ \ldots \circ N_{i_2}(\tdun{$i_1$}).
\end{align*}
When $|\D|=1$, these are the generators of the Connes-Moscovici subalgebra of \cite{Connes98}.

\begin{prop}
Let $\h_{CM}^\D$ be the subalgebra of $\h_{\CK}^\D$ generated by all the elements $X_{i_1,\ldots,i_k}$. Then $\h_{CM}^\D$ is a Hopf subalgebra.
\end{prop}

\begin{proof} Note that $N_d$ is a derivation; as $N_d(X_{i_1,\ldots,i_k})=X_{i_1,\ldots,i_k,d}$ for all $i_1,\ldots,i_k,d\in \D$,
$\h_{CM}^\D$ is stable under $N_d$ for any $d\in \D$. As the $X_{i_1,\ldots,i_k}$ are homogeneous of degree $k$,
\[X_{i_1,\ldots,i_k} \bullet 1=kX_{i_1,\ldots,i_k}.\]
Hence, $\h_{CM}^\D$ is stable under the derivation $D:x\mapsto x\bullet 1$. We obtain 
\begin{align}
\label{EQ9}
\Delta(X_{i_1,\ldots,i_k})&=\Delta(X_{i_1,\ldots,i_{k-1}}\bullet \tdun{$i_k$}\:)\\
\nonumber&=X_{i_1,\ldots,i_{k-1}}^{(1)}\otimes X_{i_1,\ldots,i_{k-1}}^{(2)}\bullet \tdun{$i_k$}\:\\
\nonumber&+X_{i_1,\ldots,i_{k-1}}^{(1)}\bullet \tdun{$i_k$}\: \otimes X_{i_1,\ldots,i_{k-1}}^{(2)}
+X_{i_1,\ldots,i_{k-1}}^{(1)}\bullet \emptyset\otimes X_{i_1,\ldots,i_{k-1}}^{(2)} \tdun{$i_k$}\:.
\end{align} 
An easy induction on $k$ proves that $\Delta(X_{i_1,\ldots,_k})$ belongs to $\h_{CM}^\D \otimes \h_{CM}^\D$. \end{proof}

\begin{prop}
We assume that $\D$ is finite.Then $\h_{CM}^\D$ is the graded dual of the enveloping algebra of the augmentation ideal
of the Com-PreLie algebra $T(V,f)$, where $V=\vect(\D)$ and $f=\id_V$.
\end{prop}

\begin{proof}
We put $W=\vect(X_{i_1,\ldots,i_k}\mid k\geq 1, i_1,\ldots,i_k\in \D)$. As this is the case for $\h_{\CK}^\D$, for any $x\in W$,
\[\Delta(x)-x\otimes 1+1\otimes x\in W\otimes \h_{CM}^\D.\]
This implies that the graded dual of $\h_{CM}^\D$ is the enveloping of a graded algebra $\g$; as a vector space,
$\g$ is identified with $W^*$ and its preLie product is dual of the bracket $\delta$ defined on $W$ by $(\pi_W\otimes \pi_W)\circ \Delta$,
where $\pi_W$ is the canonical projection on $W$ which vanishes on $(1)+(\h_{CM}^\D)_+^2$.
By (\ref{EQ9}), using Sweedler's notation $\delta(x)=x^{(1)}\otimes x^{(2)}$, we obtain 
\begin{align*}
\delta(X_{i_1,\ldots,i_{k+1}})&=X^{(1)}_{i_1,\ldots,i_k}\otimes X^{(2)}_{i_1,\ldots,i_k}\bullet X_{i_{k+1}}
+X^{(1)}_{i_1,\ldots,i_k}\bullet X_{i_{k+1}}\otimes X^{(2)}_{i_1,\ldots,i_k}+k X_{i_1,\ldots,i_k}\otimes X_{i_{k+1}}.
\end{align*}
We shall use the following notations. If $I\subseteq [k]$, we put 
\begin{itemize}
\item $m(I)=\max(i\mid [i]\subseteq I)$, with the convention $m(I)=0$ if $1\notin I$.
\item $X_{i_I}=X_{i_{p_1},\ldots i_{p_l}}$ if $I=\{p_1<\ldots<p_l\}$.
\end{itemize}
An easy induction proves that
\begin{align*}
&\forall i_1,\ldots,i_k\in \D,&\delta(X_{i_1,\ldots,i_k})&=\sum_{\emptyset \subsetneq I \subseteq [k]} m(I) X_{i_I}\otimes X_{i_{[k]\setminus I}}.
\end{align*}
We identify $W^*$ and $T(V)_+$ via the pairing given by
\begin{align*}
&\forall i_1,\ldots,i_k,j_1,\ldots,j_l \in \D,&\langle X_{i_1,\ldots,i_k} ,j_1\ldots j_l\rangle=\delta_{(i_1,\ldots,i_k),(j_1,\ldots,j_l)}.
\end{align*}
The preLie product on $T(V)_+$ induced by $\delta$ is then given by 
\begin{align*}
i_1\ldots i_k \bullet i_{k+1}\ldots i_{k+l}&=\sum_{\sigma \in \sh(k,l)} m_k(\sigma) i_{\sigma^{-1}(1)}\ldots i_{\sigma^{-1}(k+l)}.
\end{align*}
By (\ref{EQshuffle}), this is precisely the preLie product of $T(V,f)$. \end{proof}

\begin{remark} The following map is a bijection:
\begin{align*}
\theta_{k,l}:\left\{\begin{array}{rcl}
\sh(k,l)&\longrightarrow&\sh(l,k)\\
\sigma&\longrightarrow&(k+l\: k+l-1\ldots 1)\circ \sigma \circ (k+l\: k+l-1\ldots 1).
\end{array}\right.
\end{align*}
Moreover, for any $\sigma \in \sh(k,l)$,
\[m_l(\theta_{k,l}(\sigma))=\min \{i\in l\in \{k+1,\ldots,k+l\}\mid \sigma(i)=i,\ldots,\sigma(k+l)=\sigma(k+l)\}=m_l'(\sigma),\]
with the convention $m_l'(\sigma)=0$ if $\sigma(k+l)\neq k+l$. Then the Lie bracket associated to $\bullet$ is given by 
\begin{align*}
&\forall i_1,\ldots,i_{k+l}\in \D,&[i_1\ldots i_k,i_{k+1}\ldots i_{k+l}]&=\sum_{\sigma \in \sh(k,l)} (m_k(\sigma)-m_l'(\sigma))
i_{\sigma^{-1}(1)}\ldots i_{\sigma^{-1}(k+l)}.
\end{align*}\end{remark}

\subsection{A rigidity theorem for Com-PreLie bialgebras}

\begin{theo}\label{theo27}
Let $(A,m,\bullet,\Delta)$ be a connected Com-PreLie bialgebra. If $f_A$ (defined in Proposition \ref{prop4}) is surjective, 
then $(A,m,\Delta)$ and $(T(\prim(A)),\shuffle,\Delta)$ are isomorphic Hopf algebras.
\end{theo}

\begin{proof}  We put $V=\prim(A)$. \\

\textit{First step.} As $f_A$ is surjective, there exists $g:V\longrightarrow V$ such that $f_A\circ g=\id_V$. For all $x\in V$, we put 
\[L_x:\left\{\begin{array}{rcl}
A&\longrightarrow&A\\
y&\longrightarrow&g(x)\bullet y.
\end{array}\right.\]
For all $y\in A$,
\[\Delta \circ L_x(y)=\emptyset\otimes g(x)\bullet y+g(x)\bullet y^{(1)}\otimes y^{(2)}=\emptyset\otimes L_x(y)+(\id \otimes L_x)\circ \Delta(y).\]
Hence, $L_x$ is a $1$-cocycle of $A$. Moreover, $L_x(1)=g(x)\bullet 1=f_A\circ g(x)=x$.
For all $x_1,\ldots,x_n \in V$, we define $\omega(x_1,\ldots,x_n)$ inductively on $n$ by 
\[\omega(x_1,\ldots,x_n)=\begin{cases}
\emptyset\mbox{ if }n=0,\\
L_{x_1}(\omega(x_2,\ldots,x_{n-1}))\mbox{ if }n \geq 1.
\end{cases}\]
In particular, $\omega(v)=v$ for all $v\in V$. An easy induction proves that 
\[\Delta(\omega(x_1,\ldots,x_n))=\sum_{i=0}^n \omega(x_1,\ldots,x_i) \otimes \omega(x_{i+1},\ldots,x_n).\]
Hence, the following map is a coalgebra morphism:
\[\omega:\left\{\begin{array}{rcl}
T(V)&\longrightarrow&A\\
x_1\ldots x_n&\longrightarrow&\omega(x_1,\ldots,x_n).
\end{array}\right.\]
It is injective: if $\ker(\omega)$ is nonzero, then it is a nonzero coideal of $T(V)$, so it contains nonzero primitive elements of $T(V)$,
that is to say nonzero elements of $V$. For all $v\in V$, $\omega(v)=L_v(1)=v$: contradiction. Let us prove that $\omega$ is surjective.
As $A$ is connected, for any $x\in A_+$, there exists $n\geq 1$ such that $\tdelta^{(n)}(x)=0$. Let us prove that $x \in Im(\omega)$
by induction on $n$. If $n=1$, then $x \in V$, so $x=\omega(x)$. Let us assume the result at all ranks $<n$. By coassociativity of $\tdelta$,
$\tdelta^{(n-1)}(x)\in V^{\otimes n}$. We put $\tdelta^{(n-1)}(x)=x_1\otimes \ldots \otimes x_n \in V^{\otimes n}$.
Then $\tdelta^{(n-1)}(x)=\tdelta^{(n-1)}(\omega(x_1,\ldots,x_n))$. By the induction hypothesis,
$x-\omega(x_1,\ldots,x_n) \in Im(\omega)$, so $x\in Im(\omega)$. \\

We proved that the coalgebras $A$ and $T(V)$ are isomorphic. We now assume that $A=T(V)$ as a coalgebra.\\

\textit{Second step.} We denote by $\pi$ the canonical projection on $V$ in $T(V)$. Let $\varpi:T_+(V)\longrightarrow V$ be any linear map.
We define 
\[F_\varpi:\left\{\begin{array}{rcl}
T(V)&\longrightarrow&T(V)\\
x_1\ldots x_n&\longrightarrow&\displaystyle \sum_{k=1}^n \sum_{i_1+\ldots+i_k=n} \varpi(x_1\ldots x_{i_1})\ldots
\varpi(x_{i_1+\ldots+i_{k-1}+1}\ldots x_n).
\end{array}\right.\]
Let us prove that $F_\varpi$ is the unique coalgebra endomorphism such that $\pi \circ F_\varpi=\varpi$.
Firstly,
\begin{align*}
\Delta(F_\varpi(x_1\ldots x_n))&=\sum_{i_1+\ldots +i_k=n}
\Delta(\varpi(x_1\ldots x_{i_1})\ldots
\varpi(x_{i_1+\ldots+i_{k-1}+1}\ldots x_n))\\
&=\sum_{i_1+\ldots +i_k=n}\sum_{j=0}^k
\varpi(x_1\ldots x_{i_1})\ldots \varpi(x_{i_1+\ldots+i_{j-1}+1}\ldots x_{i_1+\ldots+i_j})\\
&\otimes \varpi(x_{i_1+\ldots+i_j+1}\ldots x_{i_1+\ldots i_{j+1}})\ldots \varpi(x_{i_1+\ldots+i_{k-1}+1}\ldots x_n))\\
&=\sum_{i=0}^n F_\varpi(x_1\ldots x_i) \otimes F_\varpi(x_{i+1}\ldots x_n)\\
&=(F_\varpi \otimes F_\varpi)\circ \Delta(x_1\ldots x_n).
\end{align*}
Moreover,
\begin{align*}
\pi \circ F_\varpi(x_1\ldots x_n)&=\sum_{k=1}^n \sum_{i_1+\ldots+i_k=n}\pi( \varpi(x_1\ldots x_{i_1})\ldots
\varpi(x_{i_1+\ldots+i_{k-1}+1}\ldots x_n))\\
&=\pi\circ \varpi(x_1\ldots x_n)+0\\
&=\varpi(x_1\ldots x_n).
\end{align*}
Let us now prove the unicity. Let $F,G$ be two coalgebra endomorphisms such that $\pi \circ F=\pi\circ G=\varpi$.
If $F\neq G$, let $x_1\ldots x_n$ be a word of $T(V)$, such that $F(x_1\ldots x_n)-G(x_1\ldots x_n)\neq 0$, of minimal length.
By minimality of $n$,
\[\tdelta(F(x_1\ldots x_n))=(F\otimes F)\circ \tdelta(x_1\ldots x_n)=(G\otimes G)\circ \tdelta(x_1\ldots x_n)
=\tdelta(G(x_1\ldots x_n)).\]
Hence, $F(x_1\ldots x_n)-G(x_1\ldots x_n) \in \prim(T(V))=V$, so
\[F(x_1\ldots x_n)-G(x_1\ldots x_n)=\pi(F(x_1\ldots x_n)-G(x_1\ldots x_n))
=\varpi(x_1\ldots x_n)-\varpi(x_1\ldots x_n)=0.\]
This is a contradiction, so $F=G$.\\

\textit{Third step.} Let $\varpi_1,\varpi_2:T_+(V)\longrightarrow V$ and let $F_1=F_{\varpi_1}$, $F_2=F_{\varpi_2}$ be the associated
coalgebra morphisms. Then 
\[\pi \circ F_2\circ F_1(x_1\ldots x_n)
=\sum_{i_1+\ldots+i_k=n} \varpi_2(\varpi_1(x_1\ldots x_{i_1})\ldots \varpi_1(x_{i_1+\ldots+i_{k-1}+1}\ldots x_n)).\]
We denote this map by $\varpi_2\diamond \varpi_1$.  By the unicity in the second step, $F_2\circ F_1=F_{\varpi_2\diamond \varpi_1}$.
It is not difficult to prove that for any $\varpi:T_+(V)\longrightarrow V$, there exists $\varpi':T_+(V)\longrightarrow V$, such that
$\varpi'\diamond \varpi=\varpi \diamond \varpi'=\pi$ if, and only if, $\varpi_{\mid V}$ is invertible. 
If this holds, then $F_{\varpi}\circ F_{\varpi'}=F_{\varpi'}\circ F_{\varpi}=F_{\pi}=\id$, by the unicity in the second step.
So, if $\varpi_{\mid V}$ is invertible, then $F_{\varpi}$ is invertible. \\

\textit{Fourth step.} We denote by $*$ the product of $T(V)$. Let us choose $\varpi:T_+(V)\longrightarrow V$ such that $\varpi(T_+(V)*T_+(V))=(0)$.
Let $F=F_\varpi$ the associated coalgebra morphism. As $\emptyset$ is the unique group-like element of $T(V)$, the unit of $*$ is $\emptyset$.
Let us prove that for all $x,y\in T(V)$, $F(x*y)=F(x)\cdot F(y)$.  We proceed by induction on $length(x)+length(y)=n$.
As $\emptyset$ is the unit for both $*$ and $\cdot$ and $F(\emptyset)=\emptyset$, it is obvious if $x$ or $y$ is equal to $\emptyset$:
this observation covers the case $n=0$. Let us assume the result at all rank $<n$. By the preceding observation on the unit, we can assume
that $x,y\in T_+(V)$. We put $G=F\circ *$ and $H=\cdot \circ (F\otimes F)$.
They are both coalgebra morphisms from $T(V)\otimes T(V)$ to $T(V)$. Moreover,
\[\pi \circ G(x\otimes y)=\pi\circ F(x*y)=\varpi(x*y)=0.\]
As the shuffle product is graded for the length, $\pi \circ H(x\otimes y)=0$. By the induction hypothesis,
\[\tdelta\circ G(x\otimes y)=(G\otimes G)\circ \tdelta(x\otimes y)=(F\otimes F)\circ \tdelta(x\otimes y)
=\tdelta\circ F(x\otimes y).\]
Hence, $G(x\otimes y)-F(x\otimes y)$ is primitive, so belongs to $V$. This implies 
\[G(x\otimes y)-F(x\otimes y)=\pi(G(x\otimes y)-F(x\otimes y))=0-0=0.\]
So $F(x*y)=G(x\otimes y)=F(x\otimes y)=F(x)\shuffle F(y)$. 
Hence, $F$ is a bialgebra morphism from $(T(V),*,\Delta)$ to $(T(V),\shuffle,\Delta)$.\\

By the third and fourth steps, in order to prove that $(T(V),*,\Delta)$ and $(T(V),\shuffle,\Delta)$ are isomorphic,
it is enough to find $\varpi:T_+(V)\longrightarrow V$, such that $\varpi_{\mid V}$ is invertible and $\varpi(T_+(V)*T_+(V))=(0)$;
hence, it is enough to prove that $V \cap (A_+ *A_+)=(0)$. \\

\textit{Last step.}  We define $\Delta:\morp(A)\longrightarrow \morp(A\otimes A,A)$ by $\Delta(f)(x \otimes y)=f(x*y)$. 
We denote by $\star$ the convolution product of $\morp(A)$ induced by the bialgebra $(A,*,\Delta)$. Let $f,g \in \morp(A)$.
We assume that we can write $\Delta(f)=f^{(1)}\otimes f^{(2)}$ and $\Delta(g)=g^{(1)}\otimes g^{(2)}$, that is to say, for all $x,y \in A$,
\begin{align*}
f(xy)&=f^{(1)}(x)*f^{(2)}(y), &g(xy)&=g^{(1)}(x)*g^{(2)}(y).
\end{align*}
Then, as $*$ is commutative,
\begin{align*}
f\star g(x*y)&=f(x^{(1)}*y^{(1)})*g(x^{(2)}*y^{(2)})\\
&=f^{(1)}(x^{(1)})*f^{(2)}(y^{(1)})*g^{(1)}(x^{(2)}) *g^{(2)}(y^{(2)})\\
&=f^{(1)}(x^{(1)})*g^{(1)}(x^{(2)})*f^{(2)}(y^{(1)}) *g^{(2)}(y^{(2)})\\
&=f^{(1)}\star g^{(1)}(x)* f^{(1)}\star g^{(2)}(y).
\end{align*}
Hence, $\Delta(f\star g)=\Delta(f)\star \Delta(g)$.

Let $\rho$ be the canonical projection on $A_+$ and $1$ be the unit of  the convolution algebra $\morp(V)$. 
Then $1+\rho=\id$. As $\Delta(\id)=\id\otimes \id$ and $\Delta(1)=1\otimes 1$, this gives 
\[\Delta(\rho)=\rho\otimes 1+1\otimes \rho+\rho\otimes \rho.\]
We consider 
\[\psi=\ln(1+\rho)=\sum_{n=1}^\infty\frac{(-1)^{n+1}}{n} \rho^{\star n}.\]
As $A$ is connected, for all $x \in A$, $\rho^{\star n}(x)=0$ if $n$ is great enough, so $\psi$ exists. Moreover,
as $\Delta$ is compatible with the convolution product,
\begin{align*}
\Delta(\psi)&=\ln(1\otimes 1+\rho\otimes 1+1\otimes \rho+\rho\otimes \rho)\\
&=\ln((1+\rho)\otimes(1+\rho))\\
&=\ln(1+\rho)\otimes 1)+\ln(1\otimes(1+\rho))\\
&=\ln(1+\rho)\otimes 1+1\otimes \ln(1+\rho)\\
&=\psi \otimes 1+1\otimes \psi.
\end{align*}
We used $((1+\rho) \otimes 1)\star (1\otimes (1+\rho))=(1\otimes (1+\rho))\star((1+\rho)\otimes 1)
=(1+\rho)\otimes (1+\rho)$ for the third equality. Hence, for all $x,y\in A$,
\[\psi(x*y)=\psi(x)\varepsilon(y)+\varepsilon(x)\psi(y).\]
In particular, if $x,y \in A_+$, $\psi(x*y)=0$.  If $x \in V$, then $\rho^1(x)=x$ and if $n\geq 2$,
\[\rho^{*n}(x)=\sum_{i=1}^n \rho(1)*\ldots *\rho(1)*\rho(x)*\rho(1)*\ldots*\rho(1)=0.\]
So $\psi(x)=x$. Finally, if $x \in V \cap (A_+ *A_+)$, $\psi(x)=x=0$. So $V \cap (A_+*A_+)=(0)$. \end{proof}

The following result is proved for $\h_{\CK}^\D$ in \cite{Broadhurst2000} and in \cite{Foissy3}:

\begin{cor}
The Hopf algebras $\CP(\D)$ and $\h_{\CK}^\D$ are isomorphic to shuffle algebras.
\end{cor}

\begin{proof}
$\CP(\D)$ is a connected Com-PreLie bialgebra. Moreover, if $x \in \CP(\D)$, homogeneous of degree $n$, $x \bullet \emptyset=n x$.
Hence, as the homogeneous component of degree $0$ of $\prim(\CP(\D))$ is zero, $f_{\CP(\D)}$ is invertible.
By the rigidity theorem, $\CP(\D)$ is, as a Hopf algebra, isomorphic to a shuffle algebra. The proof is similar for $\h_{\CK}^\D$. \end{proof}

\begin{remark} \begin{enumerate}
\item This is not the case for $\UCP(\D)$. For example, if $d,e$ are two distinct elements of $\D$, it is not difficult to prove that there is no
element $x\in \UCP(\D)$ such that 
\[\Delta(x)=x\otimes 1+1\otimes x+\tdun{$(0,d)$}\hspace{5mm}\otimes \tdun{$(0,e)$}\hspace{4mm}.\]
So $\UCP(\D)$ is not cofree.
\item $\CP(\D)$ and $\h_{\CK}^\D$ are not isomorphic, as Com-PreLie bialgebras, to any $T(V,f)$. Indeed, in $T(V,f)$, for any $x\in V$
such that $f(x)=x$, $x\shuffle x=2x\bullet x=2xx$. In $\CP(\D)$ or $\h_{\CK}^\D$, for any $d\in \D$, with $x=\tdun{$d$}$,
$f(x)=x$ but $x\cdot x\neq 2x\bullet x$.
\end{enumerate}\end{remark}

\subsection{Dual of $\UCP(\D)$ and $\CP(\D)$}

We identify $\UCP(\D)$ and its graded dual by considering the basis of partitioned trees as orthonormal.
Similarly, we identify $\CP(\D)$ and $\h_{\CK}^D$ with their graded dual.\\

Let us consider the Hopf algebra $(\UCP(\D),\cdot,\Delta)$. As a commutative algebra, it is freely generated by the set
$\UPT_1(\D)$ of partitioned trees decorated by $\N\times \D$ with one root. Moreover, if $T\in \UPT_1(\D)$,
\[\Delta(T)-1\otimes T\in \vect(\UPT_1(\D))\otimes \UCP(\D).\]
Consequently, this is a right-sided combinatorial bialgebra in the sense of \cite{Loday2010}, 
and its graded dual is the enveloping algebra of a preLie algebra $\g_{\UCP}(\D)$.  Direct computations prove the following result:

\begin{theo} \label{theo31}
The preLie algebra $\g_{\UCP}(\D)$ is the linear span of $\UPT_1(\D)$. For any $T,T'\in \UPT_1(\D)$, the PreLie product is given by 
\[T\diamond T'=\sum_{\substack{s\in V(T), \\b\in \bl(s)\sqcup\{*\}}} (T\bullet_{s,b} T')[-1]_s.\]
\end{theo}

\begin{example} If $\D=\{1\}$, forgetting the second decoration of the vertices, in $\g_{\UCP}(\D)$,
\begin{align*}
\tdun{$i$}\diamond \tdun{$j$}&=(1-\delta_{i,0}) \tddeux{$i-1$}{$j$}\hspace{4mm},\\
\tddeux{$i$}{$j$}\diamond \tdun{$k$}&=(1-\delta_{j,0})\tdtroisdeux{$i$}{$j-1$}{$k$}\hspace{4mm}
+(1-\delta_{i,0})\left(\tdtroisun{$i-1$}{$k$}{$j$}\hspace{2mm}+\hdtroisun{$i-1$}{$k$}{$j$}\hspace{2mm}\right).
\end{align*}\end{example}

Similarly, the Hopf algebra $(\CP(\D),\cdot,\Delta)$ is, as a commutative algebra, freely generated by the set
$\PT_1(\D)$ of partitioned trees decorated by $\D$ with one root. Moreover, if $T\in \PT_1(\D)$,
\[\Delta(T)-1\otimes T\in \vect(\PT_1(\D))\otimes \CP(\D).\]
Consequently, its graded dual is the enveloping algebra of a preLie algebra $\g_{\CP}(\D)$, described by the following theorem:

\begin{theo}
The preLie algebra $\g_{\CP}(\D)$ is the linear span of $\PT_1(\D)$. For any $T,T'\in \PT_1(\D)$, the PreLie product is given by 
\[T\diamond T'=\sum_{\substack{s\in V(T), \\b\in \bl(s)\sqcup\{*\}}} T\bullet_{s,b} T'.\]
\end{theo}

\begin{example} If $\D=\{1\}$, forgetting the decorations, in $\g_{\CP}(\D)$,
\begin{align*}
\tun\diamond \tun&=\tdeux,&\tdeux \diamond \tun&=\ttroisdeux+\ttroisun+\htroisun.
\end{align*}\end{example}

\begin{notation} Let $T\in \PT_1(\D)$. We can write $T=\tdun{$d$}\bullet (T_1\times \ldots \times T_k)=B_d(T_1\ldots T_k)$, 
where $T_1,\ldots,T_k \in \PT(\D)$. Up to a change of indexation,
we will always assume that $T_1,\ldots,T_p\in \PT_1(\D)$ and $T_{p+1},\ldots,T_k\notin \PT_1(\D)$. The integer $p$ is denoted by $\varsigma(T)$.
\end{notation}

\begin{prop}
As a preLie algebra, $\g_{\CP}(\D)$ is freely generated by the set of trees $T\in \PT_1(\D)$ such that $\varsigma(T)=0$.
\end{prop}

\begin{proof} We define a coproduct on $\g_{\CP}(\D)$ by
\begin{align*}
&\forall T=B_d(T_1\ldots T_k)\in \PT_1(\D),&\delta(T)&=\sum_{i=1}^{\varsigma(T)}B_d(T_1\ldots \widehat{T_i}\ldots T_k)\otimes T_i.
\end{align*}
This coproduct is permutative: indeed, 
\begin{align*}
(\delta \otimes \id)\circ \delta(T)=&\sum_{1\leq i\neq j\leq \varsigma (T)} B_d(T_1\ldots \widehat{T_i}\ldots \widehat{T_j}\ldots T_k)\otimes T_i\otimes T_j,
\end{align*}
so $(\delta \otimes \id)\circ \delta=(23).(\delta \otimes \id)\circ \delta$. Let $T=B_d(T_1\ldots T_k), T'\in \PT_1(\D)$.
Then 
\[T\diamond T'=B_d(T'T_1\ldots T_k)+\sum_{i=1}^k B_d(T_1\ldots (T_i\diamond T')\ldots T_k)+\sum_{i=1}^k B_d(T_1\ldots (T_i\shuffle T')\ldots T_k).\]
Hence,
\begin{align*}
\delta(T\otimes T')&=B_d(T_1\ldots T_k)\otimes T'+\sum_{i=1}^{\varsigma(T)}B_d(T'T_1\ldots \widehat{T_i}\ldots T_k)\otimes T_i\\
&+\sum_{i=1}^k \sum_{\substack{j=1\\j\neq i}}^{\varsigma(T)}B_d(T_1\ldots \widehat{T_j}\ldots (T_i\diamond T')\ldots T_k)\otimes T_j
+\sum_{i=1}^{\varsigma(T)} B_d(T_1\ldots \widehat{T_i}\ldots T_k)\otimes T_i\diamond T'\\
&+\sum_{i=1}^k \sum_{\substack{j=1\\j\neq i}}^{\varsigma(T)}B_d(T_1\ldots \widehat{T_j}\ldots (T_i\shuffle T')\ldots T_k)\otimes T_j\\
&=\sum_{j=1}^{\varsigma(T)}\left(B_d(T'T_1\ldots \widehat{T_j}\ldots T_k)+\sum_{\substack{i=1\\i\neq j}}^k
B_d(T_1\ldots \widehat{T_j}\ldots (T_i\diamond T'+T_i\shuffle T')\ldots T_k)\right)\otimes T_j\\
&+\sum_{i=1}^{\varsigma(T)} B_d(T_1\ldots \widehat{T_i}\ldots T_k)\otimes T_i\diamond T'+T\otimes T'\\
&=\sum_{j=1}^{\varsigma(T)}B_d(T_1\ldots \widehat{T_j}\ldots T_k)\bullet T'\otimes T_j
+\sum_{i=1}^{\varsigma(T)} B_d(T_1\ldots \widehat{T_i}\ldots T_k)\otimes T_i\diamond T'+T\otimes T'\\
&=T^{(1)}\diamond T'\otimes T^{(2)}+T^{(1)}\otimes T^{(2)}\diamond T'+T\otimes T'.
\end{align*}
By Livernet's rigidity theorem \cite{Livernet2006}, $\g_{\CP}(\D)$ is freely generated, as a preLie algebra, by $\ker(\delta)$. \\

We define
\[\Upsilon:\left\{\begin{array}{rcl}
\g_{\CP}(\D)\otimes \g_{\CP}(\D)&\longrightarrow&\g_{\CP}(\D)\\
T\otimes T'&\longrightarrow& T\bullet_{r(T),*}T',
\end{array}\right.\]
where $r(T)$ is the root of $T$. In other words, $\Upsilon(B_d(T_1\ldots T_k)\otimes T')=B_d(T'T_1\ldots T_k)$;
this implies that for any $T\in \PT_1(\D)$, $\Upsilon \circ \delta(T)=\varsigma(T) T$.
Hence, if $x=\sum a_T T\in \ker(\delta)$, $\Upsilon \circ \delta(x)=\sum a_T\varsigma(T)T=0$, so $x$ is a linear span of trees $T$ such that 
$\varsigma(T)=0$. The converse is trivial. \end{proof}

We denote by $PT_1^{(0)}(\D)$ the set of partitioned trees $T\in \PT_1(\D)$ with $\varsigma(T)=0$. 
The preceding Proposition implies that the Hopf algebras $(\CP(\D),\cdot,\Delta)$ and $\left(\h_{\CK}^{\PT_1^{(0)}(\D)},m,\Delta\right)$ are isomorphic.
We obtain an explicit isomorphism between them:

\begin{defi}
Let $T\in \PT(\D)$ and $\pi=\{P_1,\ldots,P_k\}$ be a partition of $V(T)$. We shall write $\pi\triangleleft T$ if the following condition holds:
\begin{itemize}
\item For all $i\in [k]$, the partitioned rooted forest $T_{\mid P_i}$, denoted by $T_i$, belongs to $\PT_1^{(0)}(\D)$. 
\end{itemize}
If $\pi \triangleleft T$, the contracted graph $T/\pi$ is a rooted forest (one forgets about the blocks of $T$).
The vertex of $T/\pi$ corresponding to $P_i$ is decorated by $T_i$, making $T/\pi$ an element of $\T(\PT_1^{(0)}(\D))$.
\end{defi}

\begin{cor}
The following map is a Hopf algebra isomorphism:
\[\Theta:\left\{\begin{array}{rcl}
(\CP(\D),\cdot,\Delta)&\longrightarrow&\left(\h_{\CK}^{\PT_1^{(0)}(\D)},\cdot,\Delta\right)\\
T\in \PT(\D)&\longrightarrow&\displaystyle \sum_{\pi\triangleleft T}T/\pi.
\end{array}\right.\]
\end{cor}

\begin{example} If $\D=\{1\}$, forgetting the decorations, with $a=\tun$ and $b=\htroisun$,
\begin{align*}
\Theta(\tun)&=\tdun{$a$},&\Theta(\tdeux)&=\tddeux{$a$}{$a$},&
\Theta(\ttroisun)&=\tdtroisun{$a$}{$a$}{$a$},&\Theta(\htroisun)&=\tdtroisun{$a$}{$a$}{$a$}+\tdun{$b$}.
\end{align*}\end{example}

\subsection{Extension of the preLie product $\diamond$ to all partitioned trees}

We now extend the preLie product $\diamond$ to the whole $\CP(\D)$:

\begin{prop}
We define a product on $\CP(\D)$ by
\begin{align*}
&\forall T,T'\in \PT(\D),&T\diamond T'=\sum_{\substack{s\in V(T), \\b\in \bl(s)\sqcup\{*\}}} T\bullet_{s,b} T'.
\end{align*}
Then $(\CP(\D),\diamond,\cdot)$ is a Com-PreLie algebra.
\end{prop}

\begin{proof}
Obviously, for any $x,y,z\in \PT(\D)$, $(x\cdot y)\diamond z=(x\diamond z)\cdot x+x\cdot (y\diamond z)$.
Let $T_1,T_2,T_3\in \PT(\D)$. Then 
\begin{align*}
(T_1\diamond T_2)\diamond T_3&=\sum_{\substack{s_1\in V(T_1), \\b_1\in \bl(s_1)\sqcup\{*\}}}
\sum_{\substack{s_2\in V(T_1), \\b_2\in \bl(s_2)\sqcup\{*\}}} (T_1\bullet_{s_1,b_1} T_2)\bullet_{s_2,b_2} T_3\\
&+\sum_{\substack{s_1\in V(T_1), \\b_1\in \bl(s_1)\sqcup\{*\}}}
\sum_{\substack{s_2\in V(T_2), \\b_2\in \bl(s_2)\sqcup\{*\}}} (T_1\bullet_{s_1,b_1} T_2)\bullet_{s_2,b_2} T_3\\
&=\sum_{\substack{s_1\in V(T_1), \\b_1\in \bl(s_1)\sqcup\{*\}}}
\sum_{\substack{s_2\in V(T_1), \\b_2\in \bl(s_2)\sqcup\{*\}}} (T_1\bullet_{s_1,b_1} T_2)\bullet_{s_2,b_2} T_3\\
&+\sum_{\substack{s_1\in V(T_1), \\b_1\in \bl(s_1)\sqcup\{*\}}}
\sum_{\substack{s_2\in V(T_2), \\b_2\in \bl(s_2)\sqcup\{*\}}} T_1\bullet_{s_1,b_1} (T_2\bullet_{s_2,b_2} T_3)\\
&=\sum_{\substack{s_1\in V(T_1), \\b_1\in \bl(s_1)\sqcup\{*\}}}
\sum_{\substack{s_2\in V(T_1), \\b_2\in \bl(s_2)\sqcup\{*\}}} (T_1\bullet_{s_1,b_1} T_2)\bullet_{s_2,b_2} T_3+T_1\diamond (T_2\diamond T_3).
\end{align*}
Hence,
\begin{align*}
(T_1\diamond T_2)\diamond T_3-T_1\diamond (T_2\diamond T_3)&=\sum_{\substack{s_1\in V(T_1), \\b_1\in \bl(s_1)\sqcup\{*\}}}
\sum_{\substack{s_2\in V(T_1), \\b_2\in \bl(s_2)\sqcup\{*\}}} (T_1\bullet_{s_1,b_1} T_2)\bullet_{s_2,b_2} T_3\\
&=\sum_{\substack{s_1\neq s_2\in V(T_1)\\b_1\in \bl(s_1)\sqcup\{*\},\\b_2\in \bl(s_2)\sqcup\{*\}}} (T_1\bullet_{s_1,b_1} T_2)\bullet_{s_2,b_2} T_3\\
&+\sum_{\substack{s\in V(T_1), \\b_1\neq b_2 \in \bl(s)\sqcup\{*\}}}(T_1\bullet_{s,b_1} T_2)\bullet_{s,b_2} T_3
+\sum_{\substack{s\in V(T_1), \\b \in \bl(s)\sqcup\{*\}}}(T_1\bullet_{s,b} T_2)\bullet_{s,b} T_3.
\end{align*}
The three terms of this sum are symmetric in $T_2$, $T_3$, so 
\begin{align*}
(T_1\diamond T_2)\diamond T_3-T_1\diamond (T_2\diamond T_3)&=(T_1\diamond T_3)\diamond T_2-T_1\diamond (T_3\diamond T_2).
\end{align*}
Finally, $(\CP(\D),\diamond,\cdot)$ is a Com-PreLie algebra. \end{proof}

\begin{defi}
Let $T=(t,I,d)$ and $T'=(t,I',d)$ be two elements of $\PT(\D)$ with the same underlying decorated rooted trees.
We shall say that $T\leqslant T'$ is $I'$ is a refinement of $I$.
This defines a partial order on $\PT(\D)$.
\end{defi}

\begin{example} If $a,b,c,d\in \D$, $\hdquatretrois{$a$}{$d$}{$c$}{$b$} \leqslant \hdquatreun{$a$}{$d$}{$c$}{$b$},
\hdquatreun{$a$}{$d$}{$b$}{$c$},\hdquatreun{$a$}{$c$}{$b$}{$d$}\leqslant \tdquatreun{$a$}{$d$}{$c$}{$b$}$.
\end{example}

\begin{theo}
The following map is an isomorphism of Com-PreLie algebras:
\[\Psi:\left\{\begin{array}{rcl}
(\CP(\D),\circ,\cdot)&\longrightarrow&(\CP(\D),\diamond,\cdot)\\
T\in \PT(\D)&\longrightarrow&\displaystyle \sum_{T'\leqslant T} T'.
\end{array}\right.\]
\end{theo}

\begin{proof}
As $\leqslant$ is a partial order, $\Psi$ is bijective. Let $T_1,T_2\in \PT(\D)$.

1. If $T'\leqslant T_1\cdot T_2$, let us put $T'_1=T_1\cap T'$ and $T'_2=T_2\cap T'$. Then, obviously, $T'_1\leqslant T_1$ and $T'_2\leqslant T_2$.
Moreover, $T'=T'_1\leqslant T'_2$. Conversely, if $T'_1\leqslant T_1$ and $T'_2\leqslant T_2$, then $T'_1\cdot T'_2\leqslant T_1\cdot T_2$.
Hence,
\begin{align*}
\Psi(T_1\cdot T_2)&=\sum_{T'\leqslant T_1\cdot T_2} T'
=\sum_{T'_1\leqslant T_1,\:T'_2\leqslant T_2} T'_1\cdot T'_2
=\Psi(T_1)\cdot \Psi(T_2).
\end{align*}

2. Let $s\in V(T_1)$ and $T'\leqslant T_1\bullet_{s,*} T_2$. We put $T'_1=T'\cap T_1$ and $T'_2=T'\cap T_2$.
Then, obviously, $T'_1\leqslant T_1$ and $T'_2\leqslant T_2$. If the block of roots of $T_2$ is also a block of $T'$,
then $T'=T'_1\bullet_{s,*} T'_2$. Otherwise, there exists a unique $b\in \bl(s)$ such that $T'=T'_1\bullet_{s,b} T'_2$.
Conversely, if $T'_1\leqslant T_1$, $T'_2\leqslant T_2$, $s\in V(T'_1)$ and $b \in \bl(s)\sqcup \{*\}$, then $T'_1\bullet_{s,b} T'_2\leqslant T_1\bullet_{s,*} T_2$.
Hence,
\begin{align*}
\Psi(T_1\circ T_2)&=\sum_{s\in V(T_1)}\sum_{T'\leqslant T_1\bullet_{s,*} T_2} T'\\
&=\sum_{T'_1\leqslant T_1,\:T'_2\leqslant T_2} \sum_{s\in V(T'_1),b\in \bl(s)\sqcup\{*\}} T'_1\bullet_{s,b} T'_2\\
&=\Psi(T_1)\diamond \psi(T_2).
\end{align*}
So $\Psi$ is a Com-PreLie algebra isomorphism. \end{proof}

\begin{example} In the non-decorated case,
\begin{align*}
\Psi(\tun)&=\tun,&\Psi(\ttroisdeux)&=\ttroisdeux,\\
\Psi(\tdeux)&=\tdeux,&\Psi(\tquatreun)&=\tquatreun+3\hquatreun+\hquatretrois,\\
\Psi(\ttroisun)&=\ttroisun+\htroisun,&\Psi(\hquatreun)&=\hquatreun+\hquatretrois,\\
\Psi(\htroisun)&=\htroisun,&\Psi(\hquatretrois)&=\hquatretrois.
\end{align*}\end{example}

\bibliographystyle{amsplain}
\bibliography{biblio}

\end{document}

%% file: arbres.tex
\newcommand{\tun}{\begin{picture}(5,0)(-2,-1)
\put(0,0){\circle*{2}}
\end{picture}}

\newcommand{\tdeux}{\begin{picture}(7,7)(0,-1)
\put(3,0){\circle*{2}}
\put(3,5){\circle*{2}}
\put(3,0){\line(0,1){5}}
\end{picture}}

\newcommand{\ttroisun}{\begin{picture}(15,12)(-5,-1)
\put(3,0){\circle*{2}}
\put(6,7){\circle*{2}}
\put(0,7){\circle*{2}}
\put(-0.65,0){$\vee$}
\end{picture}}

\newcommand{\ttroisdeux}{\begin{picture}(5,15)(-2,-1)
\put(0,0){\circle*{2}}
\put(0,5){\circle*{2}}
\put(0,10){\circle*{2}}
\put(0,0){\line(0,1){5}}
\put(0,5){\line(0,1){5}}
\end{picture}}

\newcommand{\tquatreun}{\begin{picture}(15,12)(-5,-1)
\put(3,0){\circle*{2}}
\put(6,7){\circle*{2}}
\put(0,7){\circle*{2}}
\put(3,7){\circle*{2}}
\put(-0.65,0){$\vee$}
\put(3,0){\line(0,1){7}}
\end{picture}}

\newcommand{\tquatredeux}{\begin{picture}(15,18)(-5,-1)
\put(3,0){\circle*{2}}
\put(6,7){\circle*{2}}
\put(0,7){\circle*{2}}
\put(0,14){\circle*{2}}
\put(-0.65,0){$\vee$}
\put(0,7){\line(0,1){7}}
\end{picture}}

\newcommand{\tquatretrois}{\begin{picture}(15,18)(-5,-1)
\put(3,0){\circle*{2}}
\put(6,7){\circle*{2}}
\put(0,7){\circle*{2}}
\put(6,14){\circle*{2}}
\put(-0.65,0){$\vee$}
\put(6,7){\line(0,1){7}}
\end{picture}}

\newcommand{\tquatrequatre}{\begin{picture}(15,18)(-5,-1)
\put(3,5){\circle*{2}}
\put(6,12){\circle*{2}}
\put(0,12){\circle*{2}}
\put(3,0){\circle*{2}}
\put(-0.65,5){$\vee$}
\put(3,0){\line(0,1){5}}
\end{picture}}

\newcommand{\tquatrecinq}{\begin{picture}(9,19)(-2,-1)
\put(0,0){\circle*{2}}
\put(0,5){\circle*{2}}
\put(0,10){\circle*{2}}
\put(0,15){\circle*{2}}
\put(0,0){\line(0,1){5}}
\put(0,5){\line(0,1){5}}
\put(0,10){\line(0,1){5}}
\end{picture}}

\newcommand{\tdun}[1]{\begin{picture}(10,5)(-2,-1)
\put(0,0){\circle*{2}}
\put(3,-2){\tiny #1}
\end{picture}}

\newcommand{\tddeux}[2]{\begin{picture}(12,5)(0,-1)
\put(3,0){\circle*{2}}
\put(3,5){\circle*{2}}
\put(3,0){\line(0,1){5}}
\put(6,-2){\tiny #1}
\put(6,3){\tiny #2}
\end{picture}}

\newcommand{\tdtroisun}[3]{\begin{picture}(20,12)(-5,-1)
\put(3,0){\circle*{2}}
\put(6,7){\circle*{2}}
\put(0,7){\circle*{2}}
\put(-0.65,0){$\vee$}
\put(5,-2){\tiny #1}
\put(9,5){\tiny #2}
\put(-5,5){\tiny #3}
\end{picture}}

\newcommand{\tdtroisdeux}[3]{\begin{picture}(12,15)(-2,-1)
\put(0,0){\circle*{2}}
\put(0,5){\circle*{2}}
\put(0,10){\circle*{2}}
\put(0,0){\line(0,1){5}}
\put(0,5){\line(0,1){5}}
\put(3,-2){\tiny #1}
\put(3,3){\tiny #2}
\put(3,9){\tiny #3}
\end{picture}}

\newcommand{\tdquatreun}[4]{\begin{picture}(20,12)(-5,-1)
\put(3,0){\circle*{2}}
\put(6,7){\circle*{2}}
\put(0,7){\circle*{2}}
\put(3,7){\circle*{2}}
\put(-0.6,0){$\vee$}
\put(3,0){\line(0,1){7}}
\put(5,-2){\tiny #1}
\put(8,5){\tiny #2}
\put(1,10){\tiny #3}
\put(-6,5){\tiny #4}
\end{picture}}

%% file: hyperarbres.tex
\newcommand{\hdeux}{\begin{picture}(12,8)(-3,-1)
\textcolor{red}{\put(0,0){\circle*{2}}
\put(5,0){\circle*{2}}
\put(0,0){\line(1,0){5}}}
\end{picture}}

\newcommand{\htroisun}{\begin{picture}(12,8)(-3,-1)
\put(3,0){\circle*{2}}
\put(-0.6,0.1){$\vee$}
\textcolor{red}{\put(6,7){\circle*{2}}
\put(0,7){\circle*{2}}
\put(0,7){\line(1,0){5}}}
\end{picture}}

\newcommand{\htroisdeux}{\begin{picture}(12,8)(-3,-1)
\put(0,5){\circle*{2}}
\put(0,0){\line(0,1){5}}
\textcolor{red}{\put(0,0){\circle*{2}}
\put(5,0){\circle*{2}}
\put(0,0){\line(1,0){5}}}
\end{picture}}

\newcommand{\htroistrois}{\begin{picture}(12,8)(-3,-1)
\put(5,5){\circle*{2}}
\put(5,0){\line(0,1){5}}
\textcolor{red}{\put(5,0){\circle*{2}}
\put(0,0){\circle*{2}}
\put(0,0){\line(1,0){5}}}
\end{picture}}

\newcommand{\htroisquatre}{\begin{picture}(17,0)(-3,-1)
\textcolor{red}{\put(0,0){\circle*{2}}
\put(5,0){\circle*{2}}
\put(10,0){\circle*{2}}
\put(0,0){\line(1,0){10}}}
\end{picture}}

\newcommand{\hquatreun}{\begin{picture}(12,12)(-3,-1)
\put(3,0){\circle*{2}}
\put(6,7){\circle*{2}}
\put(-0.5,0.1){$\vee$}
\put(3,0){\line(0,1){7}}
\textcolor{red}{\put(0,7){\line(1,0){2}}
\put(0,7){\circle*{2}}
\put(3,7){\circle*{2}}}
\end{picture}}

\newcommand{\hquatredeux}{\begin{picture}(12,12)(-3,-1)
\put(3,0){\circle*{2}}
\put(0,7){\circle*{2}}
\put(-0.5,0.1){$\vee$}
\put(3,0){\line(0,1){7}}
\textcolor{red}{\put(3,7){\line(1,0){2}}
\put(6,7){\circle*{2}}
\put(3,7){\circle*{2}}}
\end{picture}}

\newcommand{\hquatretrois}{\begin{picture}(12,12)(-3,-1)
\put(3,0){\circle*{2}}
\put(-0.5,0.1){$\vee$}
\put(3,0){\line(0,1){7}}
\textcolor{red}{\put(0,7){\line(1,0){5}}
\put(6,7){\circle*{2}}
\put(0,7){\circle*{2}}
\put(3,7){\circle*{2}}}
\end{picture}}

\newcommand{\hquatrequatre}{\begin{picture}(12,18)(-3,-1)
\put(3,0){\circle*{2}}
\put(0,14){\circle*{2}}
\put(-0.5,0.1){$\vee$}
\put(0,7){\line(0,1){7}}
\textcolor{red}{\put(0,7){\line(1,0){5}}
\put(6,7){\circle*{2}}
\put(0,7){\circle*{2}}}
\end{picture}}

\newcommand{\hquatrecinq}{\begin{picture}(12,18)(-3,-1)
\put(3,0){\circle*{2}}
\put(6,14){\circle*{2}}
\put(-0.5,0.1){$\vee$}
\put(6,7){\line(0,1){7}}
\textcolor{red}{\put(0,7){\line(1,0){5}}
\put(6,7){\circle*{2}}
\put(0,7){\circle*{2}}}
\end{picture}}

\newcommand{\hquatresix}{\begin{picture}(12,18)(-3,-1)
\put(3,5){\circle*{2}}
\put(3,0){\circle*{2}}
\put(-0.5,5.1){$\vee$}
\put(3,0){\line(0,1){5}}
\textcolor{red}{\put(0,12){\line(1,0){5}}
\put(6,12){\circle*{2}}
\put(0,12){\circle*{2}}}
\end{picture}}

\newcommand{\hquatresept}{\begin{picture}(18,8)(-5,-1)
\put(6,7){\circle*{2}}
\put(0,7){\circle*{2}}
\put(-0.6,0.1){$\vee$}
\textcolor{red}{\put(3,0){\line(1,0){5}}
\put(8,0){\circle*{2}}
\put(3,0){\circle*{2}}}
\end{picture}}

\newcommand{\hquatrehuit}{\begin{picture}(18,8)(-5,-1)
\put(9,7){\circle*{2}}
\put(3,7){\circle*{2}}
\put(2.4,0.1){$\vee$}
\textcolor{red}{\put(0,0){\line(1,0){5}}
\put(0,0){\circle*{2}}
\put(6,0){\circle*{2}}}
\end{picture}}

\newcommand{\hquatreneuf}{\begin{picture}(12,12)(-3,-1)
\put(0,10){\circle*{2}}
\put(0,5){\circle*{2}}
\put(0,0){\line(0,1){5}}
\put(0,5){\line(0,1){5}}
\textcolor{red}{\put(0,0){\line(1,0){5}}
\put(0,0){\circle*{2}}
\put(5,0){\circle*{2}}}
\end{picture}}

\newcommand{\hquatredix}{\begin{picture}(12,12)(-3,-1)
\put(5,5){\circle*{2}}
\put(5,10){\circle*{2}}
\put(5,0){\line(0,1){5}}
\put(5,5){\line(0,1){5}}
\textcolor{red}{\put(0,0){\line(1,0){5}}
\put(0,0){\circle*{2}}
\put(5,0){\circle*{2}}}
\end{picture}}

\newcommand{\hquatreonze}{\begin{picture}(18,8)(-5,-1)
\put(-0.6,0.1){$\vee$}
\textcolor{red}{\put(3,0){\line(1,0){5}}
\put(3,0){\circle*{2}}
\put(8,0){\circle*{2}}}
\textcolor{blue}{\put(-3.5,7){\line(1,0){5}}
\put(2.5,7){\circle*{2}}
\put(-3.5,7){\circle*{2}}}
\end{picture}}

\newcommand{\hquatredouze}{\begin{picture}(12,8)(-3,-1)
\put(2.4,0.1){$\vee$}
\textcolor{red}{\put(0,0){\line(1,0){5}}
\put(6,0){\circle*{2}}
\put(0,0){\circle*{2}}}
\textcolor{blue}{\put(-.5,7){\line(1,0){5}}
\put(5.5,7){\circle*{2}}
\put(-.5,7){\circle*{2}}}
\end{picture}}

\newcommand{\hquatretreize}{\begin{picture}(12,8)(-3,-1)
\put(0,5){\circle*{2}}
\put(5,5){\circle*{2}}
\put(5,0){\line(0,1){5}}
\put(0,0){\line(0,1){5}}
\textcolor{red}{\put(0,0){\line(1,0){5}}
\put(0,0){\circle*{2}}
\put(5,0){\circle*{2}}}
\end{picture}}

\newcommand{\hquatrequatorze}{\begin{picture}(17,8)(-3,-1)
\put(0,5){\circle*{2}}
\put(0,0){\line(0,1){5}}
\textcolor{red}{\put(0,0){\line(1,0){10}}
\put(0,0){\circle*{2}}
\put(5,0){\circle*{2}}
\put(10,0){\circle*{2}}}
\end{picture}}

\newcommand{\hquatrequinze}{\begin{picture}(17,8)(-2,-1)
\put(5,5){\circle*{2}}
\put(5,0){\line(0,1){5}}
\textcolor{red}{\put(0,0){\line(1,0){10}}
\put(0,0){\circle*{2}}
\put(5,0){\circle*{2}}
\put(10,0){\circle*{2}}}
\end{picture}}

\newcommand{\hquatreseize}{\begin{picture}(17,8)(-3,-1)
\put(10,5){\circle*{2}}
\put(10,0){\line(0,1){5}}
\textcolor{red}{\put(0,0){\line(1,0){10}}
\put(0,0){\circle*{2}}
\put(5,0){\circle*{2}}
\put(10,0){\circle*{2}}}
\end{picture}}

\newcommand{\hquatredixsept}{\begin{picture}(21,8)(-3,-1)
\textcolor{red}{\put(0,0){\circle*{2}}
\put(5,0){\circle*{2}}
\put(10,0){\circle*{2}}
\put(15,0){\circle*{2}}
\put(0,0){\line(1,0){15}}}
\end{picture}}

\newcommand{\hddeux}[2]{\begin{picture}(16,8)(-5,-1)
\textcolor{red}{\put(0,0){\circle*{2}}
\put(5,0){\circle*{2}}
\put(0,0){\line(1,0){5}}}
\put(-5,-2){\tiny #1}
\put(7,-2){\tiny #2}
\end{picture}}

\newcommand{\hdtroisun}[3]{\begin{picture}(20,12)(-5,-1)
\put(3,0){\circle*{2}}
\put(-0.6,0.1){$\vee$}
\put(5,-2){\tiny #1}
\put(9,5){\tiny #2}
\put(-5,5){\tiny #3}
\textcolor{red}{\put(6,7){\circle*{2}}
\put(0,7){\circle*{2}}
\put(0,7){\line(1,0){5}}}
\end{picture}}

\newcommand{\hdquatreun}[4]{\begin{picture}(20,12)(-5,-1)
\put(3,0){\circle*{2}}
\put(6,7){\circle*{2}}
\put(-0.5,0.1){$\vee$}
\put(3,0){\line(0,1){7}}
\textcolor{red}{\put(0,7){\line(1,0){2}}
\put(0,7){\circle*{2}}
\put(3,7){\circle*{2}}}
\put(5,-2){\tiny #1}
\put(8,5){\tiny #2}
\put(1,10){\tiny #3}
\put(-6,5){\tiny #4}
\end{picture}}

\newcommand{\hdquatretrois}[4]{\begin{picture}(20,12)(-5,-1)
\put(3,0){\circle*{2}}
\put(-0.5,0.1){$\vee$}
\put(3,0){\line(0,1){7}}
\textcolor{red}{\put(0,7){\line(1,0){5}}
\put(6,7){\circle*{2}}
\put(0,7){\circle*{2}}
\put(3,7){\circle*{2}}}
\put(5,-2){\tiny #1}
\put(8,5){\tiny #2}
\put(1,10){\tiny #3}
\put(-6,5){\tiny #4}
\end{picture}}